\documentclass{article}
\usepackage{psfrag}
\usepackage[parfill]{parskip}
\usepackage{float}
\usepackage{graphicx}
\usepackage{epsfig}
\usepackage{amsmath, amsthm, amssymb}
\usepackage{verbatim}
\usepackage{multicol}
\usepackage{url}
\usepackage{latexsym}

\usepackage{mathrsfs}
\usepackage[colorlinks = true,
            linkcolor  = blue,
            citecolor  =blue,
            urlcolor   =blue,
            bookmarks  = true]{hyperref}
 \usepackage{enumitem}
\setlist[enumerate,1]{label=(\roman*)}
\usepackage[normalem]{ulem}
\usepackage{bm}
\usepackage{dsfont}
\usepackage{stmaryrd}
\usepackage[dvipsnames]{xcolor}  

\newcommand{\address}[1]{\begin{center}\small #1\end{center}}
\usepackage[backend=biber, style=numeric, citestyle=numeric-comp]{biblatex}
\DeclareFieldFormat{labelnumberwidth}{\mkbibbrackets{#1}}
\DeclareNameAlias{author}{family-given}

\DeclareFieldFormat*{title}{#1}

\AtEveryBibitem{%
  \clearfield{issn}
  \clearfield{url}
  \clearfield{eprint}
}

\renewbibmacro{in:}{}

\DeclareFieldFormat[article]{volume}{\textbf{#1}}
\DeclareFieldFormat[article]{number}{\textbf{#1}}
\renewbibmacro*{volume+number+eid}{%
  \printfield{volume}
  \iffieldundef{number}
    {}%
    {\printtext[parens]{\printfield{number}}}
  \setunit{\addcomma\space}%
  \printfield{eid}}

\DeclareFieldFormat[journal]{journaltitle}{\mkbibemph{#1}}

\DeclareNameAlias{author}{family-given}
\DeclareNameAlias{editor}{family-given}

\DeclareDelimFormat{namedelim}{\addcomma\space}
\DeclareFieldFormat[article,incollection,inproceedings,book]{giveninit}{\mkbibnamegiven{#1}}

\DeclareNameFormat{family-given}{%
  \usebibmacro{name:family-given}
    {\namepartfamily}
    {\namepartgiveni} 
    {\namepartprefix}
    {\namepartsuffix}%
  \usebibmacro{name:andothers}}

\DeclareFieldFormat{pages}{#1}

\numberwithin{equation}{section}

\theoremstyle{plain}

\newtheorem{mydef}{Definition}[section]
\newtheorem{myrem}{Remark}[section]
\newtheorem{Theorem}{Theorem}[section]
\newtheorem{lemma}{Lemma}[section]
\newtheorem{proposition}{Proposition}[section]

\def\author#1{\par
    {\centering{\authorfont#1}\par\vspace*{0.05in}}
}

\def\titlefont{\fontsize{13}{15}\bfseries\boldmath\selectfont\centering{}}
\def\authorfont{\fontsize{13}{15}}

\let\affiliationfont\rhfont

\def\address#1{\par
    {\centering{\affiliationfont#1\par}}\par\vspace*{11pt}
}
\def\title#1{
    \thispagestyle{plain}
    \vspace*{-14pt}
    \vskip 79pt
    {\centering{\titlefont #1\par}}%
    \vskip 1em
}

\begin{document}

\title{The smallest singular value of signed random combinatorial matrices}

\vspace{1cm}

\author{Kexin Yu}
\address{Shandong University\\kexinyu01@gmail.com}

\vspace{0.3cm}

\begin{abstract}
Let \(M_n\) be an \(n\times n\) signed random combinatorial matrix whose rows are independent and uniformly distributed over the set of \(\{-1,0,1\}\)-vectors with exactly \(n/2\) zero coordinates. Despite the dependence induced by the row constraints, we prove that there exist constants $C,c > 0$ such that for any $\varepsilon\ge0$, 
\begin{align*} 
\textbf{P}\left(s_{n}(M_n)\le {\varepsilon}{n^{-1/2}}\right)\le C\varepsilon+e^{-cn}. 
\end{align*} 
In particular, the probability that \(M_n\) is singular is exponentially small. Our approach builds on the Combinatorial Least Common Denominator (CLCD) introduced by Tran and develops the method in the present constrained setting.
\end{abstract}

\section{Introduction and Main Results}\label{sec1}
Let $A_n$ be an $n\times n$ real matrix with singular values $s_1(A_n)\ge\cdots\ge s_n(A_n)$. In particular, the largest and smallest singular values of $A_n$ are defined to be 
\begin{align*}
s_1(A_n) = \sup_{x\in \mathbb{S}^{n-1}}\|A_n x\|_2, \quad s_n(A_n) = \inf_{x\in \mathbb{S}^{n-1}}\|A_n x\|_2,
\end{align*}
where $\|\cdot\|_2$ denotes the Euclidean norm on $\mathbb{R}^n$, and $\mathbb{S}^{n-1}$ is the unit Euclidean sphere in $\mathbb{R}^n$.
\par
In computer science, the condition number of $A_n$, i.e., the ratio of $s_1(A_n)$ and $s_n(A_n)$, is an important indicator to measure the stability of a linear system $A_n x=b$ to input perturbations. 
Bounding the condition number of $A_n$ is an important problem in this field. The behavior of $s_1(A_n)$ is well understood under general assumptions on the matrix entries $\xi$. In \cite{MR950344}, Yin, Bai, and Krishnaiah proved that, when $\textbf{E} \xi=0$ and $\textbf{E}|\xi|^4<\infty$, with high probability
\begin{align*}
s_1(A_n) \sim \sqrt{n}.
\end{align*}
Hence, estimating the typical magnitude of the condition number reduces to studying the lower bound of the smallest singular value $s_n(A_n)$, i.e., to establishing quantitative invertibility estimates for $A_n$.
\par
For random matrices with independent and identically distributed (i.i.d.) standard Gaussian entries, the magnitude of $s_n(A_n)$ is of the order $1 / \sqrt{n}$ with high probability. This observation goes back to von Neumann \cite{MR157875}, and it was proved by Edelman \cite{doi:10.1137/0609045} that
\begin{align*}
\textbf{P}\left(s_n\left(A_n\right) \leq \varepsilon n^{-1 / 2}\right) \sim \varepsilon,\qquad \varepsilon\ge 0.
\end{align*}

In the discrete i.i.d. case, in particular for $\{\pm1\}$-valued entries, the main difficulty is arithmetic: a linear form $\sum_i \xi_i v_i$ may have atypically large concentration only when $v$ has strong additive structure.  Tao and Vu introduced the inverse Littlewood--Offord theory as a systematic framework to quantify this principle in \cite{MR2480613}.

Following earlier work of Rudelson \cite{MR2434885} and Tao and Vu \cite{MR2480613}, Rudelson and Vershynin developed a general geometric approach for small ball probability estimates in \cite{MR2407948}. Their results apply to square random matrices with i.i.d. mean-zero, unit-variance, sub-Gaussian entries, and imply the following bound on the smallest singular value:
\begin{align*}
\textbf P\big(s_{n}(A_n)\le \varepsilon n^{-1/2}\big)\le C \varepsilon + e^{-c n},\qquad \varepsilon\ge 0.
\end{align*}
Their argument decomposes the unit sphere into compressible and incompressible vectors and reduces the analysis of incompressible vectors to small ball probability estimates for linear forms of independent coordinates. 
The (inverse) Littlewood--Offord theory supplies the needed anti-concentration input. It provides effective bounds on the small ball probabilities of weighted sums of independent random variables. This is the key ingredient in establishing the lower bound of $s_{n}(A_n)$.

A major breakthrough was obtained by Tikhomirov \cite{MR4076632}, who proved that for an $n\times n$ matrix $B_n$ with i.i.d. Bernoulli entries,
\begin{align*}
\textbf{P}\left(s_n\left(B_n\right) \leq \varepsilon n^{-1 / 2}\right) \leq C \varepsilon+\left(\frac{1}{2}+o_n(1)\right)^n, \qquad \varepsilon\ge 0.
\end{align*}
This result settles the long-standing conjecture on the correct exponential scale of the singularity probability. 

Although the smallest singular value problem for random matrices with i.i.d. entries is comparatively well developed, this problem becomes significantly more difficult when one considers models of random matrices with dependencies between entries. 
A basic and widely studied example is the symmetric
$\{\pm1\}$ ensemble, where the upper triangular entries are independent but the symmetry constraint $M_n=M_n^\top$ induces global dependence. 
In \cite{MR4810062}, Campos, Jenssen, Michelen, and Sahasrabudhe proved that the singularity probability of this model is exponentially small.

Beyond symmetric ensembles, an important class of random matrix models with dependent entries is given by models subject to regularity constraints. Because of the dependence among the entries induced by the regularity constraints, it is generally not convenient to describe the model solely in terms of the marginal distributions of individual entries. Instead, it is more natural to view the model at the level of the matrix ensemble.
Let $M_{n,d}$ denote the set of all $n\times n$ matrices with entries in $\{0,1\}$ whose row and column sums are all equal to $d$ (equivalently, adjacency matrices of $d$-regular directed graphs), and let $A_n$ be a uniformly chosen element of $M_{n,d}$. 

In the dense regime $\min(d,n-d)\ge \lambda n$, Jain, Sah, and Sawhney proved a sharp lower bound for the smallest singular value in \cite{MR4523249}: there exist constants $C_\lambda,c_\lambda>0$ such that
\begin{align*}
\textbf{P}\left(s_{n}(A_n)\le \varepsilon n^{-1/2}\right)\le C_\lambda\varepsilon+2e^{-c_\lambda n},\qquad \varepsilon\ge 0.
\end{align*}
More specifically, their argument built on Tran's Combinatorial Least Common Denominator (CLCD) approach and the associated small ball probability estimates, and developed a suitable variant adapted to the $d$-regular setting.

In \cite{tran2020smallestsingularvaluerandom}, Tran introduced the CLCD in the setting of random combinatorial matrices. He considered an $n\times n$ random matrix $Q_n$ with entries in $\{0,1\}$ whose rows are independent and uniformly distributed over vectors with exactly $n/2$ ones (equivalently, $Q_n$ is the biadjacency matrix of a random bipartite graph in which each left vertex has degree exactly $n/2$). He obtained the following exponential upper bound on the singularity probability: 
\begin{align*}
\textbf{P}\left(s_{n}(Q_n)\le {\varepsilon}{n^{-1/2}}\right)\le C\varepsilon+2e^{-c n},
\qquad \varepsilon\ge 0.
\end{align*}
This study was further developed in the work of Jain, Sah, and Sawhney \cite{MR4356701}, in which they obtained a sharper bound and confirmed \cite[Conjecture 1.4]{MR4356701}:
\begin{align*}
\textbf{P}\left(s_{n}(Q_n)\le {t}{n^{-1/2}}\right)\le C_{\varepsilon}t+\left(\frac12+\varepsilon\right)^n,
\qquad \text{for all } t\ge 0,\,\varepsilon> 0 \text{ and } C_{\varepsilon}\text{ depending on }\varepsilon.
\end{align*}
More recently, Li, Litvak, and Yu \cite{MR5042290} obtained a complementary upper bound for the smallest singular value of dense random combinatorial matrices, showing that the smallest singular value of $Q_n$ is typically of order $n^{-1/2}$.

In a related direction, Cook introduced a signed analogue of the random regular digraph adjacency model in \cite{MR3602844}. 
Let $\mathcal M_{n,d}^{\pm}$ be the set of $n\times n$ matrices $M_{\pm}$ with entries in $\{-1,0,1\}$ such that for every $k\in[n]$,
\begin{align*}
\sum_{i=1}^n |M_{\pm}(i,k)|=\sum_{j=1}^n |M_{\pm}(k,j)|=d.
\end{align*}
If $M_{\pm}$ is uniform on $\mathcal M_{n,d}^{\pm}$ and $C\log^2 n\le d\le n$, then
\begin{align*}
\textbf P\big(\det(M_{\pm})=0\big)=O\!\left(d^{-1/4}\right).
\end{align*}
To better understand Cook's model, a useful representation is $M_{\pm}\stackrel{d}{=} M\circ \Xi$, where $M$ is the $d-$regular digraph adjacency matrix and $\Xi$ is an i.i.d. sign matrix independent of $M$. Here $\circ$ denotes the Hadamard (or Schur) product, so that $M_{\pm}(i,j)\stackrel{d}= M(i,j)\Xi(i,j)$ for each $i,j \in [n]$, in the sense that these two random matrices have the same distribution as random elements of $\{-1,0,1\}^{n\times n}$.

Motivated by these developments, we study random matrices with independent rows, each of which has exactly $n/2$ non-zero entries taking values in $\{\pm1\}$, while the remaining $n/2$ entries are zero. In other words, the numbers of $1$ and $-1$ in each row are not fixed separately, only their sum is prescribed. For this constrained ensemble, we establish a lower bound for the smallest singular value.

As mentioned above, our main result concerns signed matrices whose row vectors are independent and each have exactly $n/2$ non-zero coordinates. Let $\mathcal{M}_{n, n/2}^{ \pm}$ denote the set of $n \times n$ matrices $Q_{\pm}$ with entries in $\{-1,0,1\}$ satisfying
\begin{align*}
    \sum_{j=1}^n\left|Q_{\pm}(i, j)\right|=\frac{n}{2},
\end{align*}
for all $i \in[n]$. 

Let $M_{n}$ be a uniform random element of $\mathcal{M}_{n, n/2}^{ \pm}$. Then $M_n$ admits the distributional representation $M_{n}\stackrel{d}{=} Q_n\circ \Xi$, where $Q_n$ is an $n\times n$ random matrix with entries in $\{0,1\}$ whose rows are independent and uniformly distributed over vectors with exactly $n/2$ ones and $\Xi$ is an i.i.d. sign matrix independent of $Q_n$.

We now state our main result.
\begin{Theorem}\label{main}
Let $n \in \mathbb{N}$ be even. There exist absolute constants $C,c > 0$ such that for all sufficiently large $n$, and for all $\varepsilon\ge0$,
\begin{align*}
\textbf{P}\left(s_n\left(M_n\right) \leq \varepsilon n^{-1/2}\right) \leq C \varepsilon+ e^{-c n}.
\end{align*}

\end{Theorem}
\begin{myrem}
For $\varepsilon=0$, Theorem \ref{main} implies that the singularity probability is exponentially small. In particular,
\begin{align*}
\textbf{P}\left(M_n \text { is singular}\right) \leq  e^{-c n}.
\end{align*}
\end{myrem}

\begin{myrem}
Theorem \ref{main} remains valid in the more general setting where each row has exactly $d$ non-zero entries, with $\min(d, n-d) \ge \lambda n$ for some fixed $\lambda>0$. We restrict attention to the case of even $n$ and $d = n/2$ only for simplicity of exposition.

Let $R_1, \cdots, R_n$ denote the independent rows of the matrix $M_n$. Direct calculation confirms that the covariance matrix of the row vectors is $\mathrm{Cov}(R_i)=\textbf{E}[R_i^{\top} R_i] = \frac d n I_n$. In contrast with i.i.d. signed matrices, whose row covariance is $I_n$, the present model has row covariance $\frac d nI_n$. Thus, by varying the number of non-zero entries, one can tune the covariance to different scalar multiples of the identity within this family.
\end{myrem}

The following theorem is the main ingredient in the proof of our main result. Before stating this theorem, we recall the definition of the Lévy concentration function.
\begin{mydef}
The Lévy concentration function of a real-valued random variable $X$ at scale $\varepsilon \ge 0$ is defined by
\begin{align*}
\mathcal{L}(X, \varepsilon) :=\sup_{\lambda\in\mathbb{R}}\textbf{P}(|X-\lambda|\le\varepsilon).
\end{align*}
\end{mydef}
\begin{Theorem}\label{theorem:distance}
Let $M_n$ be an $n \times n$ random matrix satisfying the assumptions of Theorem \ref{main}. Let $R_1, \dots, R_n$ denote its rows, and for each $j \in [n]$, let $H_j$ denote the subspace spanned by the remaining rows:
\begin{align*}
H_j := \mathrm{span}\{ R_i : i \neq j, i \in [n] \}.
\end{align*}
Let $v_j$ be a random unit vector orthogonal to $H_j$ and measurable with respect to the $\sigma$–field generated by $H_j$. Then there exist constants $C, c > 0$ such that for all $\varepsilon \ge 0$, the Lévy concentration function satisfies
\begin{align}\label{eq:Levy}
\mathcal{L}(\langle R_j, v_j \rangle, \varepsilon) \le C\varepsilon + 2e^{-cn}.
\end{align}
In particular, this implies the distance bound
\begin{align*}
\textbf{P} \left( \operatorname{dist}(R_j, H_j) \le \varepsilon \right) \le C\varepsilon + 2e^{-cn}.
\end{align*}
\end{Theorem}

We now deduce Theorem \ref{main} from Theorem \ref{theorem:distance}. As in \cite{MR2407948}, the first step is to decompose the unit sphere into compressible and incompressible vectors.
\begin{mydef}[Compressible and incompressible vectors]
Fix $\delta, \rho \in (0, 1)$. A vector $x \in \mathbb{R}^n$ is called \textit{sparse} if $|\operatorname{supp}(x)| \leq \delta n$. A vector $x \in \mathbb{S}^{n-1}$ is called \textit{compressible} if $x$ is within Euclidean distance $\rho$ from the set of all sparse vectors. A vector $x \in \mathbb{S}^{n-1}$ is called \textit{incompressible} if it is not compressible. The sets of sparse, compressible, and incompressible vectors will be denoted by $\operatorname{Sparse}, \operatorname{Comp}$ and $\operatorname{Incomp}$ respectively.
\end{mydef}

\noindent Using the decomposition of the unit sphere $\mathbb{S}^{n-1} = \operatorname{Comp}(\delta, \rho) \cup \operatorname{Incomp}(\delta, \rho)$, we break the invertibility problem into two subproblems, for compressible and incompressible vectors:
\begin{align*}
\textbf{P} \left( s_n(M_n) \le \varepsilon n^{-1/2} \right) \le \textbf{P} \left( \inf_{x \in \operatorname{Comp}(\delta, \rho)} \|x^{\top}M_n\|_2 \le \varepsilon n^{-1/2} \right)+ \textbf{P} \left( \inf_{x \in \operatorname{Incomp}(\delta, \rho)} \|x^{\top}M_n\|_2 \le \varepsilon n^{-1/2} \right).
\end{align*}
The compressible part admits an exponentially small bound. For the incompressible vectors, we use a version of the “invertibility via distance” bound from \cite{MR2407948}, which holds for any $n\times n$ random matrix $M_n$:
\begin{align*}
\textbf{P}\left(\inf _{x \in \operatorname{Incomp}(\delta, \rho)}\left\|x^{\top}M_n\right\|_2 \leq \varepsilon \frac{\rho}{\sqrt{n}}\right) \leq \frac{1}{\delta n} \sum_{k=1}^n \textbf{P}\left(\operatorname{dist}\left(R_k, H_k\right) \leq \varepsilon\right) ,
\end{align*}
where $R_1, \ldots, R_n$ are the row vectors of $M_n$, and $H_k$ is the subspace spanned by all rows except the $k$-th one.
\begin{myrem}
On the event $s_n(M_n)\le \varepsilon n^{-1/2}$, there exists $v\in\mathbb S^{n-1}$ such that $\|M_n v\|_2\le \varepsilon n^{-1/2}$. Since $M_n$ and its transpose have the same singular values, this is equivalent to the existence $x \in \mathbb{S}^{n-1}$ such that $\|x^\top M_n\|_2 \leq \varepsilon n^{-1/2}$. We use this left-sided formulation because the rows of $M_n$ are independent.
\end{myrem}
Thus, proving distance-type theorems, such as Theorem \ref{theorem:distance}, is crucial for studying the problem of smallest singular value of random matrices. In \cite{MR2407948}, Rudelson and Vershynin obtained a result analogous to Theorem \ref{theorem:distance} under the assumption that the entries are independent and identically distributed. They introduced the Least Common Denominator (LCD) to measure how close a suitably scaled vector is to the integer lattice. Using the LCD, they estimated the Lévy concentration function.

In the dependent setting where $R_j$ contains $n/2$ zeros and $n/2$ ones, in \cite{tran2020smallestsingularvaluerandom}, Tran developed a combinatorial version of the least common denominator (CLCD) to estimate $\mathcal{L}(\langle R_j, v_j \rangle, \varepsilon)$ and handle the dependence among the coordinates. In this paper, we extend the CLCD framework to the present signed setting.

\begin{mydef}[{Combinatorial Least Common Denominator, \cite[Definition 1.4]{tran2020smallestsingularvaluerandom}}]
For a vector $v \in \mathbb{R}^n$, and parameters $\alpha, \gamma>0$, define

\begin{align*}
    \operatorname{CLCD}_{\alpha, \gamma}(v):=\inf \left\{\theta>0: \operatorname{dist}\left(\theta \mathrm{D}(v), \mathbb{Z}^{\binom{n}{2}}\right)<\min (\gamma\|\theta  \mathrm{D}(v)\|_2, \alpha)\right\} .
\end{align*}
Here by $\mathrm{D}(v)$ we denote the vector in $\mathbb{R}^{\binom{n}{2}}$ whose $(i, j)$-coordinate is $v_i-v_j$, for $1 \leq i<j \leq n$.
\end{mydef}

\begin{myrem}
A key ingredient in Tran's work \cite{tran2020smallestsingularvaluerandom} is a method based on combinatorial statistics. Specifically, to exploit the analytic framework of Hal\'{a}sz \cite{MR494478} and Roos \cite{MR4904157}, Tran introduced the random sum
\begin{align*}
    W_{v} := \eta_1 v_1 + \dots + \eta_n v_n,
\end{align*}
where the vector $(\eta_1, \dots, \eta_n)^{\top}$ is chosen uniformly from $\{0, 1\}^n$ subject to the strict linear constraint $\sum_{i=1}^n \eta_i = n/2$. However, in our model, the sum of the number of 1 and -1 is fixed, but the number of 1 and -1 itself is not fixed. Consequently, we necessitate the construction of a new form of $W_{v}$ adapted to this structure, requiring us to re-establish its connection with the arithmetic structure of the vector.
\end{myrem}
In our model, we construct the random sum defined as:
\begin{align*}
    W_{v} := \xi_1 v_1 + \dots + \xi_n v_n.
\end{align*}
where the vector $(\xi_1, \dots, \xi_n)^{\top}$ is chosen uniformly from the set of vectors in $\{-1,0, 1\}^n$ satisfying the constraint $\sum_{i=1}^n |\xi_i| = n/2$.
Let $S\subset\{1,\dots,n\}$ be a uniformly random subset with $|S|=n/2$, and let $\sigma_1,\dots,\sigma_n$ be independent Rademacher random variables, independent of $S$, so that $$\textbf{P}(\sigma_i = 1) = \textbf{P}(\sigma_i = -1) = 1/2,\qquad i\in[n].$$ 

Then
$(\xi_1,\dots,\xi_n)^\top \stackrel d= (\sigma_1\mathbf 1_{\{1\in S\}},\dots,\sigma_n\mathbf 1_{\{n\in S\}})^\top,$
which implies that
\begin{align*}
W_v \stackrel d= \sum_{i=1}^n \sigma_i v_i\,\mathbf 1_{\{i\in S\}}.
\end{align*}

Therefore, the randomness of $W_v$ in our model is more involved, since it comes not only from the random signs but also from the random choice of the support, which makes the estimation of small ball probabilities substantially more delicate. Inspired by the work of Tran \cite{tran2020smallestsingularvaluerandom}, we establish the following relation between the CLCD of $v$ and the anti-concentration properties of $W_v$. This may be viewed as a variant of the Littlewood--Offord theory.

\begin{Theorem}\label{LOT}
For any $b>0$ and $\gamma \in(0,1)$ there exists $C>0$ depending only on $b$ and $\gamma$ for which the following holds. Let $v \in \mathbb{R}^n$ such that $\|\mathrm{D}(v)\|_2 \geq b \sqrt{n}$. Then for every $\alpha>0$ and $\varepsilon \geq 0$, we have

\begin{align*}
\mathcal{L}\left(W_{v}, \varepsilon\right) \leq C \varepsilon+\frac{C}{\operatorname{CLCD}_{\alpha, \gamma}(v)}+C e^{-2 \alpha^2 / n}.
\end{align*}
\end{Theorem}

Note that the requirement $\|\mathrm{D}(v)\|_2 \geq b \sqrt{n}$ holds for a ``typical" vector $v$.  Given this theorem, it is natural to prove Theorem \ref{main} using the geometric approach of Rudelson and Vershynin \cite{MR2407948}.

\subsection{Notation}
The inner product in $\mathbb{R}^n$ is denoted by $\langle \cdot, \cdot \rangle$, and the Euclidean norm is denoted by $\|\cdot\|_2$. The Euclidean unit ball and sphere in $\mathbb{R}^n$ are denoted by $B_2^n$ and $\mathbb{S}^{n-1}$, respectively. Given $S, P \subset \mathbb{R}^n$, we use the standard notation $N(S, P)$ for the least number of translates of $P$ needed to cover $S$. We write $[n]$ for the set $\{1, 2, \ldots, n\}$, and $\binom{X}{k}$ is the family of all $k$-element subsets of a set $X$. For any vector $v \in \mathbb{R}^n$ and any subset $I \subseteq [n]$, let $v_I$ denote the restriction of $v$ to the coordinates indexed by $I$.

\noindent Suppose that $A = (a_{ij})_{1 \le i \le m, 1 \le j \le n}$ is an $m \times n$ real matrix, and write $A(i,j):=a_{ij}$. The operator norm of $A$ is defined as $\|A\| := \sup_{v \in \mathbb{S}^{n-1}} \|Av\|_2$, while the Hilbert-Schmidt norm of $A$ is given by $\|A\|_{\text{HS}} := \left(\sum_{i,j} a_{ij}^2\right)^{1/2}$. We also define the Orlicz $\psi_2$-norm (sub-Gaussian norm) and $\psi_1$-norm (subexponential norm) of a real-valued random variable $X$ by $\|X\|_{\psi_2} := \inf\left\{ t>0 : \textbf{E}\exp\left(\frac{X^2}{t^2}\right) \le 2 \right\}$, $\|X\|_{\psi_1} := \inf\left\{ t>0 : \textbf{E}\exp\left(\frac{|X|}{t}\right) \le 2 \right\}.$

\noindent Throughout this paper, we make extensive use of asymptotic notation. For functions \(f, g\), \(f = O_\alpha(g)\) (or \(f \lesssim_\alpha g\)) means \(f \leq C_\alpha g\) for some constant \(C_\alpha\) depending only on \(\alpha\); \(f = \Omega_\alpha(g)\) (or \(f \gtrsim_\alpha g\)) means \(f \geq c_\alpha g\) for some constant \(c_\alpha > 0\) depending only on \(\alpha\). For parameters \(\varepsilon, \delta\), the relation \(\varepsilon \ll_\alpha \delta\) means that \(\varepsilon\) is smaller than \(c_\alpha(\delta)\) for a sufficiently decaying function $c_\alpha$ depending on $\alpha$. All logarithms are natural unless specified otherwise, and floor and ceiling functions are omitted when their inclusion does not affect the argument.

\noindent We will also use $C, c, c_1, c_2$, etc. to denote unspecified positive constants whose values may differ at each occurrence and are understood to be absolute unless explicitly stated otherwise.

\subsection{Organization}
The remainder of this paper is organized as follows. In Section \ref{sec2}, we present some preliminaries; the new results are Lemma \ref{le23}, Proposition \ref{prop:first_singular_value}, and Lemma \ref{FFV}. In Section \ref{sec3}, we introduce a variant of the Littlewood--Offord theory based on the CLCD. Then we prove Theorem \ref{theorem:distance} in Section \ref{sec4} and prove Theorem \ref{main} in Section \ref{sec5}.

\section{Preliminaries}\label{sec2}
In this section, we introduce some lemmas that will be used to prove our main result. We recall a concentration inequality for sums of independent, mean-zero, subexponential random variables (see, e.g., \cite[Theorem 2.8.1]{MR3837109}).
\begin{Theorem}[Bernstein inequality]\label{Bernstein} Let $Y_1, \ldots, Y_m$ be independent, mean-zero, subexponential random variables. Then, for every $t \geq 0$, we have
\begin{align*}
\textbf{P}\left(\left|\sum_{i=1}^m Y_i\right| \geq t\right) \leq 2 \exp \left[-c_1 \min \left(\frac{t^2}{\sum_{i=1}^m\left\|Y_i\right\|_{\psi_1}^2}, \frac{t}{\max _i\left\|Y_i\right\|_{\psi_1}}\right)\right],
\end{align*}
where $c_1>0$ is an absolute constant.
\end{Theorem}

\begin{lemma}\label{le23}
Let $n \in \mathbb{N}$ be even and let \(M\) be an \(m\times n\) random matrix, $1\le m \le n$, whose rows are independent and uniformly distributed over the set of \(\{-1,0,1\}\)-vectors with exactly \(n/2\) zero coordinates. Fix  $v \in \mathbb{S}^{n-1}$. Then, for every $t \geq 0$, 
\begin{align*}
\textbf{P}\left(\left|\|M v\|_2^2-\frac{m}{2} \right| \geq t\right) \leq 2 \exp \left[-c_2 \min \left(\frac{t^2}{ n}, t\right)\right],   
\end{align*}
where $c_2>0$ is an absolute constant.
\end{lemma}

\begin{proof}[Proof of Lemma \ref{le23}]
Fix an even integer $n\ge 2$. Since $M\stackrel{d}{=} Q\circ \Xi$, where $Q$ is an $m\times n$ random matrix with entries in $\{0,1\}$ whose rows are independent and uniformly distributed over vectors with exactly $n/2$ ones and $\Xi$ is an i.i.d.\ sign matrix independent of $Q$, each row vector $\xi=(\xi_1,\dots,\xi_n)\in\{-1,0,1\}^n$ of matrix $M$ has the same distribution as
\begin{align*}
(\xi_1,\dots,\xi_n) \stackrel d= (\sigma_1\mathbf 1_{\{1\in S\}},\dots,\sigma_n\mathbf 1_{\{n\in S\}}),
\end{align*}
where $\sigma_1,\dots,\sigma_n$ are i.i.d. Rademacher random variables, and the support set $S$ is uniformly distributed over all subsets of $[n]$ of cardinality $n/2$, and $S$ is independent of $\{\sigma_i\}_{i=1}^n$.
In particular, for any fixed $v\in\mathbb S^{n-1}$, the corresponding random sum can be written as
\begin{align*}
X:=\sum_{i=1}^n v_i\xi_i
\stackrel{d}=\sum_{i=1}^n v_i\sigma_i\mathbf 1_{\{i\in S\}}.
\end{align*}
Then for all $j \in[m]$, the random variables $X_j:=\left\langle M v, \boldsymbol{e}_j\right\rangle$ are i.i.d. copies of $X$. Thus
\begin{align*}
\|Mv\|_2^2=X_1^2+\ldots+X_m^2.
\end{align*}
Since $\mathbf E[\sigma_i]=0$ and $\mathbf 1_{\{i\in S\}}$ is independent of $\sigma_i$, we have
\begin{align*}
\mathbf E[\xi_i]
=\mathbf E[\sigma_i\mathbf 1_{\{i\in S\}}]
=\mathbf E[\sigma_i]\mathbf E[\mathbf 1_{\{i\in S\}}]
=0,
\qquad i\in[n].
\end{align*}
Therefore, $\mathbf E[X]=\sum_{i=1}^n v_i\mathbf E[\xi_i]=0.$

Moreover,
\begin{align*}
\mathbf E[X^2]
=\mathbf E\left[\left(\sum_{i=1}^n v_i\xi_i\right)^2\right] 
=\sum_{i=1}^n v_i^2\mathbf E[\xi_i^2]
+\sum_{i\neq j} v_iv_j\mathbf E[\xi_i\xi_j].
\end{align*}
Since $\xi_i^2=\sigma_i^2\mathbf 1_{\{i\in S\}}=\mathbf 1_{\{i\in S\}}$,
we obtain
\begin{align*}
\mathbf E[\xi_i^2]
=\mathbf E[\mathbf 1_{\{i\in S\}}]=\textbf P(i\in S)=\frac12.
\end{align*}
For $i\neq j$, since
\begin{align*}
\mathbf E[\xi_i\xi_j]
=\mathbf E[\sigma_i\sigma_j\mathbf 1_{\{i\in S\}}\mathbf 1_{\{j\in S\}}]
=\mathbf E[\sigma_i\sigma_j]\mathbf E[\mathbf 1_{\{i\in S\}}\mathbf 1_{\{j\in S\}}]
=\mathbf E[\sigma_i]\mathbf E[\sigma_j]\mathbf E[\mathbf 1_{\{i\in S\}}\mathbf 1_{\{j\in S\}}]
=0.
\end{align*}
Consequently,
\begin{align*}
\mathbf E[X^2]=\frac12\sum_{i=1}^n v_i^2=\frac12,
\end{align*}
where in the last step we used $\|v\|_2=1$.

We next show that $X$ is sub-Gaussian.
For every $\lambda\in\mathbb R$, we have
\begin{align*}
\mathbf E e^{\lambda X}
=\mathbf E_S\left[\mathbf E_\sigma \exp\left(\lambda\sum_{i=1}^n v_i\sigma_i \mathbf 1_{\{i\in S\}}\right)\right].
\end{align*}
Here and throughout, subscripts on expectation indicate the corresponding source of randomness. For each fixed realization of $S$, the random variables $\sigma_1,\dots,\sigma_n$ are independent, hence
\begin{align*}
\mathbf E_\sigma \exp\left(\lambda\sum_{i=1}^n v_i\sigma_i\mathbf 1_{\{i\in S\}}\right)
=\prod_{i=1}^n \mathbf E_{\sigma_i} e^{\lambda v_i \sigma_i\mathbf 1_{\{i\in S\}}}.
\end{align*}
Using the classical bound for the moment generating function of a Rademacher random variable,
\begin{align*}
\mathbf E_\sigma e^{\lambda v_i\sigma_i\mathbf 1_{\{i\in S\}}}
=\cosh(\lambda v_i \mathbf 1_{\{i\in S\}})
\le \exp\left(\frac{\lambda^2 v_i^2 (\mathbf 1_{\{i\in S\}})^2}{2}\right)
=\exp\left(\frac{\lambda^2 v_i^2 \mathbf 1_{\{i\in S\}}}{2}\right),
\end{align*}
we obtain
\begin{align*}
\mathbf E_\sigma \exp\left(\lambda\sum_{i=1}^n v_i \sigma_i\mathbf 1_{\{i\in S\}}\right)
&\le \prod_{i=1}^n \exp\left(\frac{\lambda^2 v_i^2 \mathbf 1_{\{i\in S\}}}{2}\right) \\
&=\exp\left(\frac{\lambda^2}{2}\sum_{i=1}^n v_i^2 \mathbf 1_{\{i\in S\}}\right) \\
&\le \exp\left(\frac{\lambda^2}{2}\sum_{i=1}^n v_i^2\right)
=\exp\left(\frac{\lambda^2}{2}\right).
\end{align*}

Taking expectation over $S$ yields $\mathbf E e^{\lambda X}\le e^{\lambda^2/2}.$
Thus $X$ is a mean-zero sub-Gaussian random variable. Furthermore, $X^2$ is a subexponential random variable, then by the definitions of sub-Gaussian and subexponential random variables, there exists a constant $K>0$ such that 
\begin{align*}
\|X^2-\textbf E X^2\|_{\psi_1}\le \|X^2\|_{\psi_1}+\|\mathbf{E}X^2\|_{\psi_1}\le K.
\end{align*}
Let $Y_j=X_j^2-\textbf{E} X_j^2$. Then $\{Y_j\}_{j=1}^m$ are independent, mean-zero, subexponential random variables. Since $m \leq n$, applying Theorem \ref{Bernstein}, we obtain
\begin{align*}
\textbf{P}\left(\left|\|M v\|_2^2-\frac{m}{2} \right| \geq t\right) \leq 2 \exp \left[-c_2 \min \left(\frac{t^2}{ n}, t\right)\right],   
\end{align*}
where $c_2$ is an absolute constant.
\end{proof}
Then we introduce the $\varepsilon$-net argument. The following is the volumetric estimate; for details, we can see \cite{milman1986asymptotic}.
\begin{lemma}[Volumetric estimate]\label{NONet}
For any $0<\varepsilon<1$ there exists an $\varepsilon$-net $\mathcal{N} \subset {\mathbb{S}}^{n-1}$ such that
\begin{align*}
|\mathcal{N}| \leq\left(1+\frac{2}{\varepsilon}\right)^n \leq\left(\frac{3}{\varepsilon}\right)^n .
\end{align*} 
\end{lemma}

Using $\varepsilon$-nets and Lemma \ref{le23}, we prove a basic bound on the largest singular value:

\begin{proposition}\label{prop:first_singular_value}
Let $n \in \mathbb{N}$ be even and let \(M\) be an \(m\times n\) random matrix, $1\le m \le n$, whose rows are independent and uniformly distributed over the set of \(\{-1,0,1\}\)-vectors with exactly \(n/2\) zero coordinates. Then for all $t \geq C_3$, we have
\begin{align*}
\textbf{P}\left(\|M\| \geq t \sqrt{n}\right) \leq 2 \exp \left(-c_3 t^2 n\right).
\end{align*} 
\end{proposition}
\begin{proof}
Let $\mathcal{N}$ be a $(1 / 2)$-net of $\mathbb{S}^{n-1}$. For any $u \in \mathbb{S}^{n-1}$, we can choose a $v \in \mathcal{N}$ such that $\|v-u\|_2<1 / 2$. Then
\begin{align*}
\|M u\|_2 \leq\|M v\|_2+\|M\| \cdot\|v-u\|_2 \leq\|M v\|_2+\frac{1}{2}\|M\| .
\end{align*} 
This shows that $\|M\| \leq 2 \sup _{v \in \mathcal{N}}\|M v\|_2$. For every vector $v\in\mathcal{N}$, it follows from Lemma \ref{le23} that for $t\ge2$,
\begin{align*}
\textbf{P}(\|Mv\|_2 \geq t \sqrt{n} / 2) \leq \textbf{P}\left(\left|\| M v\|_2^2-\frac{m}{2}\right| \geq t^2 n / 8\right) \leq 2 e^{-c t^2 n}.
\end{align*} 
Taking the union bound and using Lemma \ref{NONet} yields
\begin{align*}
\textbf{P}\left(\|M\| \geq t \sqrt{n}\right) \leq |\mathcal{N}|\sup_{v\in\mathcal{N}}\textbf{P}\left(\|Mv\|_2 \geq t \sqrt{n} / 2\right)\le6^n\cdot2 e^{-c t^2 n}\le 2e^{-c_3t^2n}, 
\end{align*} 
provided that $t \ge C_3$ for an appropriately chosen constant $C_3 > 2$. This
completes the proof.
\end{proof}

Lemma \ref{le23} also implies the following invertibility estimate for a fixed vector.

\begin{lemma}\label{FFV}
Let $n \in \mathbb{N}$ be even and let \(M\) be an \(m\times n\) random matrix, $n/2 \leq m \leq n$, whose rows are independent and uniformly distributed over the set of \(\{-1,0,1\}\)-vectors with exactly \(n/2\) zero coordinates. There exists a constant $c_4>0$ such that the following holds for any fixed $v \in \mathbb{S}^{n-1}$
\begin{align*}
\textbf{P}\left(\|M v\|_2 \leq \sqrt{n}/4\right) \leq 2 e^{-c_4 n}.
\end{align*} 
\end{lemma}
\begin{proof}
Applying Lemma \ref{le23} with $t=3n/16$, we can obtain
\begin{align*}
\textbf{P}\left(\|Mv\|_2 \le  \sqrt{n} / 4\right) \leq \textbf{P}\left(\left|\| M v\|_2^2-\frac{m}{2}\right| \geq 3n/16\right) \leq 2 e^{-c_4 n}.
\end{align*} 
\end{proof}

\section{Anti-concentration for combinatorial statistics}\label{sec3}
Let \(a, v \in \mathbb{R}^n\) be vectors. The combinatorial statistic
\begin{align*}
W_{a,v} := a_1 v_{\pi(1)}+a_2 v_{\pi(2)}+\cdots+a_n v_{\pi(n)},
\end{align*} 
with $\pi$ a uniform random permutation of $[n]$, plays a foundational role in statistics (for an overview, see \cite{MR2103758}). Tran introduced a combinatorial variant of the LCD, which will be key to controlling the anti-concentration of $W_{a,v}$.
\begin{mydef}[{Combinatorial least common denominator, \cite[Definition 3.1]{tran2020smallestsingularvaluerandom}}]\label{def:CCLCD}
Given two vectors $a$ and $v$ in $\mathbb{R}^n$, as well as parameters $L, u>0$, the Combinatorial Least Common Denominator of the pair $(a, v)$ is

\begin{align*}
\operatorname{CLCD}_{L, u}^a(v):=\inf \left\{\theta>0: \operatorname{dist}\left(\theta  (\mathrm{D}(a) \otimes \mathrm{D}(v)), \mathbb{Z}^{\binom{n}{2}^2}\right)<\min (u\|\theta(\mathrm{D}(a) \otimes \mathrm{D}(v))\|_2, L)\right\} .
\end{align*} 
\end{mydef}

Here, by $\otimes$ we denote the tensor product. In particular, $\mathrm{D}(a) \otimes \mathrm{D}(v)$ is a vector in $\mathbb{R}^{{\binom{n}{2}}^2}$ whose $(i, j, k, \ell)$-coordinate is $\left(a_i-a_j\right)\left(v_k-v_{\ell}\right)$, for $1 \leq i<j \leq n$ and $1 \leq k<\ell \leq n$.

\begin{lemma}[{\cite[Theorem 3.2]{tran2020smallestsingularvaluerandom}}]\label{CCLCD}
Let $a$ and $v$ be two vectors in $\mathbb{R}^n$ with $\|\mathrm{D}(a) \otimes \mathrm{D}(v)\|_2 \geq b n^{3 / 2}$ for some $b>0$. Let $L>0$ and $u \in(0,1)$. Then for any $\varepsilon \geq 0$, we have
\begin{align*}
\mathcal{L}\left(W_{a, v}, \varepsilon\right) \leq C \varepsilon+\frac{C}{\operatorname{CLCD}_{L, u}^{a}(v)}+C e^{-8 L^2 / n^3}.
\end{align*}
The constant $C>0$ here depends only on $b$ and $u$.
\end{lemma}
For the detailed proof of Lemma \ref{CCLCD}, we can see \cite[Theorem 3.2]{tran2020smallestsingularvaluerandom}. We now turn to the proof of Theorem \ref{LOT}.

\begin{proof}[Proof of Theorem \ref{LOT}]
Let us recall the setting of $W_{v}$ which we considered in Section \ref{sec1}:
\begin{align}\label{W_v}
    W_{v} := \xi_1 v_1 + \dots + \xi_n v_n.
\end{align}
where the vector $(\xi_1, \dots, \xi_n)^{\top}$ is chosen uniformly from the set of vectors in $\{-1,0, 1\}^n$ satisfying the constraint $\sum_{i=1}^n |\xi_i| = n/2$.

Let $p:=|\{i\in[n]:\xi_i=1\}|$, $q:=|\{i\in[n]:\xi_i=-1\}|$, we have $p+q=n/2$. For any fixed pair $(p,q)$ with $p+q=n/2$, define the deterministic vector
\begin{align*}
a^{(p,q)}:=\big(\underbrace{1,\ldots,1}_{p},\underbrace{0,\ldots,0}_{n/2},
\underbrace{-1,\ldots,-1}_{q}\big)\in\{-1,0,1\}^n.
\end{align*}
Let $\pi$ be a uniform random permutation of $[n]$, independent of everything else, and write $\pi(a^{(p,q)})$
for the permuted vector. Then, conditional on $(p,q)$, the law of $\xi$ is uniform over all vectors in $\{-1,0,1\}^n$
with exactly $p$ entries equal to $1$, exactly $q$ entries equal to $-1$, and the remaining $n/2$ entries equal to $0$.
Consequently,
\begin{align*}
\xi\,|\, (p,q)\ \overset{d}{=}\ \pi(a^{(p,q)}).
\end{align*}
Here and in what follows, the notation $A \mid B$ is understood in the sense of the conditional law of $A$ given $B$.

Let $W_{v}$ denote the corresponding linear form (or combinatorial random sum) defined as in Theorem \ref{LOT}. For each fixed $(p,q)$, let $W_{a^{(p,q)},v}$ be denoted as 
\begin{align*}
W_{a^{(p,q)},v} := a^{(p,q)}_1 v_{\pi(1)}+a^{(p,q)}_2 v_{\pi(2)}+\cdots+a^{(p,q)}_n v_{\pi(n)}\overset{d}{=}a^{(p,q)}_{\pi(1)} v_1+a^{(p,q)}_{\pi(2)} v_2+\cdots+a^{(p,q)}_{\pi(n)} v_n.
\end{align*}
Then
\begin{align*}
W_{v}\,\big|\, (p,q)\ \overset{d}{=}\ W_{a^{(p,q)},v}\quad
\text{and for any }\lambda\in\mathbb{R},\quad
\textbf{P}(|W_{v}-\lambda|\le\varepsilon\ |\,p,q)=\textbf{P}(|W_{a^{(p,q)},v}-\lambda|\le\varepsilon).
\end{align*}
Averaging over $(p,q)$ yields
\begin{align*}
\mathcal{L}(W_{v},\varepsilon)
=\sup_{\lambda}\textbf{E}_{p,q}\,\textbf{P}(|W_{v}-\lambda|\le\varepsilon\ |\,p,q)
\le \sup_{p+q=n/2}\mathcal{L}(W_{a^{(p,q)},v},\varepsilon).
\end{align*}
Under this ordering, we set $L=\alpha n/2, \gamma=\sqrt{2}u$. $D(a)=\left(a_i-a_j\right)_{i<j}\in\{0,1,2\}$. For convenience, we set $a := a^{(p,q)}$. Then, through a simple calculation, we can obtain
\begin{align*}
|\{ (i,j) : 1 \le i < j \le n, |a_i - a_j| = 1 \}| = \frac{n^2}{4},\quad
|\{ (i,j) : 1 \le i < j \le n, |a_i - a_j| = 2 \}| = pq \in \left[0, \frac{n^2}{16}\right].
\end{align*}

The coordinates of $D(a) \otimes D(v)$ are given by $\left(a_i-a_j\right)\left(v_k-v_{\ell}\right)$,  for $1 \le i < j \le n$ and $1\le k <  \ell\le n$. Thus, after the following coordinate rearrangement (which does not modify the Euclidean norm/distance to lattice points), we obtain
\begin{align*}
D(a) \otimes D(v) \equiv (\underbrace{D(v) \oplus \cdots \oplus D(v)}_{{n^2}/{4} \text{ copies}}) \oplus (\underbrace{2 D(v) \oplus \cdots \oplus 2 D(v)}_{pq \text{ copies}}) \oplus 0\oplus\cdots \oplus 0 .
\end{align*}
\noindent Hence,
\begin{align*}
\|D(a) \otimes D(v)\|_2=\sqrt{\frac{n^2}{4}+4pq}\|D(v)\|_2\in \left[\frac{n}{2}\|D(v)\|_2,\frac{n}{\sqrt{2}}\|D(v)\|_2\right].
\end{align*}

Therefore, under the assumption of Lemma \ref{CCLCD}, since $\mathcal{L}(W_{v},\varepsilon)
\le \sup_{p+q=n/2}\mathcal{L}(W_{a,v},\varepsilon)$ and
 $\|D(a) \otimes D(v)\|_2\ge \frac{n}{2}\|D(v)\|_2 \ge \frac12b n^{3/2}$.
It suffices to show that for any $\{(p,q):p+q=n/2\}$,
\begin{align*}
{\operatorname{CLCD}_{L, u}^{a}(v)}\ge {\operatorname{CLCD}_{\alpha, \gamma}(v)},
\end{align*}
if we fix any $t\in\left\{\theta>0: \operatorname{dist}\left(\theta  (\mathrm{D}(a) \otimes \mathrm{D}(v)), \mathbb{Z}^{\binom{n}{2}^2}\right)<\min (u\|\theta(\mathrm{D}(a) \otimes \mathrm{D}(v))\|_2, L)\right\}$. Then, by the definitions of ${\operatorname{CLCD}_{L, u}^{a}(v)}$ and ${\operatorname{CLCD}_{\alpha, \gamma}(v)}$, we have
\begin{align*}
\operatorname{dist}\left(t  \mathrm{D}(v), \mathbb{Z}^{\binom{n}{2}}\right)&\le\frac{2}{n}\operatorname{dist}\left(t  (\mathrm{D}(a) \otimes \mathrm{D}(v)), \mathbb{Z}^{\binom{n}{2}^2}\right)\\
&\le\frac{2}{n}\min (u\|t(\mathrm{D}(a) \otimes \mathrm{D}(v))\|_2, L)\\
&\le\min \left(2u\sqrt{\frac{1}{4}+\frac{4pq}{n^2}}\|t\mathrm{D}(v)\|_2, \frac{2L}{n}\right)
\\
&\le\min \left(\sqrt2u\|t\mathrm{D}(v)\|_2, \frac{2L}{n}\right).
\end{align*}
Hence, we obtain ${\operatorname{CLCD}_{\frac{2L}{n}, \sqrt2u}(v)}\le {\operatorname{CLCD}_{L, u}^{a}(v)}$, we then get
\begin{align*}
\mathcal{L}\left(W_{v}, \varepsilon\right) \le C \varepsilon+\frac{C}{\operatorname{CLCD}_{\alpha, \gamma}(v)}+C e^{-2 \alpha^2 / n}.
\end{align*}
\end{proof}

\section{Proof of Theorem \ref{theorem:distance}}\label{sec4}
Recall that the condition $\|\mathrm{D}(v)\|_2 \geq b \sqrt{n}$ is required for Theorem \ref{theorem:distance} to hold. Here we briefly state that $\|\mathrm{D}(v)\|_2 \geq b \sqrt{n}$ holds for any non-almost-constant vector.
\begin{mydef}
Fix $\delta, \rho \in (0,1)$ whose values will be chosen later. A vector $v \in \mathbb{S}^{n-1}$ is called \textit{almost-constant} if one can find $\lambda \in \mathbb{R}$ such that there are at least $(1-\delta)n$ coordinates $i \in [n]$ satisfying $|v_i - \lambda| \leq \frac{\rho}{\sqrt{n}}$. A vector $v \in \mathbb{S}^{n-1}$ is called \textit{non-almost-constant} if it is not almost-constant. The sets of almost-constant and non-almost-constant vectors will be denoted by $\operatorname{Cons}(\delta,\rho)$ and $\mathcal{N}(\delta,\rho)$, respectively.
\end{mydef}
\begin{lemma}[{\cite[Lemma 2.2]{tran2020smallestsingularvaluerandom}}]\label{lemma:D(v) lower bound}
Let $\delta, \rho \in (0,1)$. Then for any vector $v \in \mathbb{S}^{n-1} \setminus \operatorname{Cons}(\delta,\rho)$, there are disjoint subsets $\sigma_1 = \sigma_1(v)\subset[n]$ and $\sigma_2 = \sigma_2(v)\subset[n]$ of cardinality $|\sigma_1|, |\sigma_2| \geq \delta n / 8$ and such that
\begin{align*}
  \frac{\rho}{\sqrt{2n}} \leq |v_i - v_j| \leq \frac{6}{\sqrt{\delta n}} \quad  \quad \text{for all } i \in \sigma_1 \text{ and } j \in \sigma_2.
\end{align*}
\end{lemma}
\noindent Then we can observe that
\begin{align*}
\|\mathrm{D}(v)\|_2\ge\sqrt{\frac{1}{2}|\sigma_1| |\sigma_2|\left(\frac{\rho}{\sqrt{2n}}\right)^2}\ge\frac{\delta\rho}{16}\sqrt{n}.
\end{align*}
Therefore, the proof of Theorem \ref{theorem:distance} can be divided into two parts: in Section \ref{subsection:Random normal is non-almost-constant}, we prove that random normal is non-almost-constant, and in Section \ref{subsection:Invertibility for non-almost-constant vectors}, we discuss the invertibility for non-almost-constant vectors.
\subsection{Random normal is non-almost-constant}\label{subsection:Random normal is non-almost-constant}
The following proposition tells us that with probability at least $1-e^{-c_5n}$, any unit vector orthogonal to $H_n$ is in $\mathcal{N}(\delta,\rho)$.
\begin{proposition}\label{prop:Random normal is non-almost-constant}
Let $\delta, \rho \in \left(0, \frac{1}{12}\right)$. Assume that $n$ is sufficiently large.
\begin{align*}
\textbf{P}(\exists v \in\operatorname{Cons}(\delta,\rho) \text{ and } v\bot H_n)\le\textbf{P}(\inf_{v\in\operatorname{Cons}(\delta,\rho)}\|M_n'v\|_2\le\frac{\sqrt{n}}{16})\le e^{-c_5n},
\end{align*}
where the $(n-1)\times n$ random matrix $M_n'$ is composed of the first $n-1$ rows of the random matrix $M_n$, i.e., the submatrix of $M_n$ obtained by removing the last row.
\end{proposition}
\begin{proof}
For any vector $v\in\operatorname{Cons}(\delta,\rho)$, by definition, there exists a scalar $\lambda \in \mathbb{R}$ and an index set $\sigma \subset [n]$ with cardinality $|\sigma| = (1 - \delta)n$ such that $|v_i - \lambda| \le \rho/\sqrt{n}$ for all $i \in \sigma$.

To discretize the range of $\lambda$, consider the grid defined by
\begin{align}\label{eq:grid_lambda}
    \mathcal{N}_\lambda = [-2, 2] \cap \frac{\rho}{\sqrt{n}}\mathbb{Z}. 
\end{align}

Any $\lambda$ that satisfies $|v_i - \lambda| \le \rho/\sqrt{n}$ can be approximated by some $\lambda_0 \in \mathcal{N}_\lambda$ that satisfies $|\lambda - \lambda_0| \le \rho/\sqrt{n}$. By the triangle inequality, this implies
\begin{align*}
    |v_i - \lambda_0| \le \frac{2\rho}{\sqrt{n}} \quad \text{for all } i \in \sigma.
\end{align*}
Next, we approximate the projection $v_{\sigma^c}$ by quantizing its coordinates uniformly with step size $\sqrt{\delta/n}$. Specifically, there exists a vector $\widetilde{u} \in \sqrt{\frac{\delta}{n}}\mathbb{Z}^{\sigma^c}$ such that $\|v_{\sigma^c} - \widetilde{u}\|_\infty \le \sqrt{\delta/n}$. 
Given that $\|v\|_2 = 1$, we have
\begin{align}\label{eq:norm_bound}
    \|\widetilde{u}\|_2 \le \|v_{\sigma^c}\|_2 + \|v_{\sigma^c} - \widetilde{u}\|_2\le \|v\|_2 + \sqrt{|\sigma^c|}  \|v_{\sigma^c} - \widetilde{u}\|_\infty \le 1 + \delta \le \frac{3}{2}.
\end{align}

We now obtain an approximate vector $u \in \mathcal{F}$ by setting $u_\sigma = (\lambda_0, \dots, \lambda_0)$ and $u_{\sigma^c} = \widetilde{u}$. 
Accordingly, we define
\begin{align*}
    \mathcal{F} := \bigcup_{|\sigma|=(1-\delta)n} \left( \left\{ \lambda_0 \mathbf{1}_\sigma : \lambda_0 \in [-2, 2] \cap \frac{\rho}{\sqrt{n}}\mathbb{Z} \right\} \oplus \left\{ \widetilde{u} \in \sqrt{\frac{\delta}{n}}\mathbb{Z}^{\sigma^c} : \|\widetilde{u}\|_2 \le 3/2 \right\} \right),
\end{align*}
using the construction of the grid in \eqref{eq:grid_lambda} and the bound \eqref{eq:norm_bound}.

From the definition of $u$, we have
\begin{align*}
    \|v - u\|_2 
    \le \|v_{\sigma^c} - \widetilde{u}\|_2 + \|v_\sigma - u_\sigma\|_2 \le \sqrt{\delta n} \|v_{\sigma^c} - \widetilde{u}\|_\infty + \sqrt{n} \max_{i \in \sigma} |v_i - \lambda_0| \le \delta + 2\rho.
\end{align*}
Let $\beta := \delta + 2\rho$, we observe that $\beta \in (0, 1/4)$. Define $\mathcal G$ to be the set of all vectors $u=u(v)$ produced by the above rounding procedure from some $v\in \operatorname{Cons}(\delta,\rho)$. Then for any $u\in {\mathcal G}$, we have $1/2\le\|u\|_2 \le 3/2$. 

It remains to estimate the cardinality of $\mathcal{F}$. The number of ways to choose the index set $\sigma$ in $\mathcal{F}$ is given by the binomial coefficient $\binom{n}{\delta n}$, which satisfies the standard bound $\binom{n}{\delta n} \le \left(\frac{e}{\delta}\right)^{\delta n}$.
For the scalar component $\lambda_0$, the size of the grid definition in \eqref{eq:grid_lambda} implies there are at most $1 + 4\sqrt{n}/\rho$ possibilities. Furthermore, a standard volumetric argument bounds the number of choices for $\widetilde{u}$ in $\mathcal{F}$ by $(5/\delta)^{\delta n}$. 

Combining these estimates, we bound the cardinality of the net $\mathcal{G}$:
\begin{align*}
|\mathcal G|\le|\mathcal{F}| \le \left(\frac{e}{\delta}\right)^{\delta n} \cdot \left(1 + \frac{4\sqrt{n}}{\rho}\right) \cdot \left(\frac{5}{\delta}\right)^{\delta n}\le e^{2\delta \log(5/\delta)n}.
\end{align*}

By Proposition \ref{prop:first_singular_value}, there exists absolute constants $K \ge 1$ such that the upper bound $\|M_n'\|  \le K\sqrt{n}$ holds with probability at least $1 - e^{-n}$. To conclude the proof, it suffices to show that for some constant $c > 0$, the probability of the event
\begin{align*}
    \mathcal{E} := \left\{ \inf_{v \in \operatorname{Cons}(\delta, \rho)} \|M_n'v\|_2 \le \frac{\sqrt{n}}{16} \quad \text{and} \quad \|M_n'\| \le K\sqrt{n} \right\}
\end{align*}
is bounded by $2e^{-cn}$. 
By Lemma \ref{FFV}, for each fixed vector $w \in \mathcal{G}$, we have
\begin{align*}
    \textbf{P}\left(\|M_n'\frac{w}{\|w\|_2}\|_2 \le \frac{\sqrt{n}}{4}\right) \le 2e^{-c_4 n}.
\end{align*}
Since $\frac12\le\|w\|_2\le\frac32$, applying a union bound over the net $\mathcal{G}$, we obtain
\begin{align}\label{eq:netF}
    \textbf{P} \left( \inf_{w \in \mathcal{G}} \|M_n'w\|_2 \le \frac{\sqrt{n}}{8} \right) \le e^{2\delta \log(5/\delta)n} \cdot 2e^{-c_4 n} \le 2e^{-c_4 n/2}, 
\end{align}
where the last inequality holds for sufficiently small $\delta$.

Suppose that $\mathcal E$ occurs. Then $\|M_n'\|  \le K\sqrt{n}$ and there exists $v \in \operatorname{Cons}(\delta, \rho)$ such that $\|M_n'v\|_2 \le \sqrt{n}/16$. By the construction of the net $\mathcal{G}$, there exists a $w \in \mathcal{G}$ satisfying $\|M_n'(v - w)\|_2 \le (\delta + 2\rho) \|M_n'\| .$
Choosing $\delta=\rho=\frac 1 {100K}$, we have
\begin{align*}
    \|M_n'w\|_2 &\le \|M_n'v\|_2 + \|M_n'(v - w)\|_2 \le \frac{\sqrt{n}}{16} + (\delta + 2\rho)K\sqrt{n} \le \frac{\sqrt{n}}{8}.
\end{align*}
Furthermore, by \eqref{eq:netF}, we obtain
\begin{align*}
   \textbf{P}(\mathcal{E})\le \textbf{P} \left( \inf_{w \in \mathcal{G}} \|M_n'w\|_2 \le \frac{\sqrt{n}}{8} \right) \le 2e^{-c_5 n},
\end{align*}
for sufficiently large $n$. This completes the proof.
\end{proof}

\subsection{Invertibility for non-almost-constant vectors}\label{subsection:Invertibility for non-almost-constant vectors}
Consequently, the proof reduces to studying the invertibility of vectors that are not almost-constant. Throughout the remainder of the proof, we set
\begin{align*}
\alpha = \mu n.
\end{align*}
For any $v \in\mathcal{N}(\delta,\rho)$ and $v\bot H_n$, if $\operatorname{CLCD}_{\mu n,\gamma}(v) \ge e^{cn}$, then the main result \eqref{eq:Levy} follows from Theorem \ref{LOT}. Therefore, the argument is reduced to proving the following proposition.

\begin{proposition} \label{prop:random_normal}
There exist constants $\mu, \gamma, c_6 \in (0, 1)$ such that for all sufficiently large $n$,
\begin{align*}
\textbf{P} \left( \exists v \in \mathcal{N}(\delta, \rho) \text{ such that }M_n'v = 0 \text{ and } \operatorname{CLCD}_{\mu n, \gamma}(v) \le e^{c_6 n} \right) \le 2^{-n}.
\end{align*}
\end{proposition}

In the following, we establish the fundamental properties of the CLCD required for the proof of Proposition \ref{prop:random_normal}. 

A crucial property of the CLCD is its approximate stability under small perturbations. This property allows us to discretize the range of possible random normal vectors.

\begin{lemma}[{Stability of CLCD, \cite[Lemma 2.14]{tran2020smallestsingularvaluerandom}}]\label{lemma:clcd_stability}
Consider a vector $v \in \mathbb{R}^n$ and parameters $\alpha > 0$, $\gamma \in (0, 1)$. For any $w \in \mathbb{R}^n$ satisfying $ \|v - w\|_2 < \frac{\gamma \|D(v)\|_2}{5\sqrt{n}},$
we have
\begin{align*}
 \operatorname{CLCD}_{\alpha/2, \gamma/2}(w) \ge \min \left\{ \operatorname{CLCD}_{\alpha, \gamma}(v), \frac{\alpha}{4\sqrt{n}\|v - w\|_2} \right\}.
\end{align*}
\end{lemma}

The next lemma provides a lower bound on the CLCD for non-almost-constant vectors.
\begin{lemma}[{Non-almost-constant vectors have large CLCD, \cite[Lemma 2.15]{tran2020smallestsingularvaluerandom}}]\label{lemma:large CLCD}
Let $\delta, \rho \in (0, 1)$, and fix $v \in \mathcal{N}(\delta, \rho)$. Then for every $\alpha > 0$ and every $\gamma \in \left(0, \frac{1}{12} \delta \rho\right)$, we have
\begin{align*}
\operatorname{CLCD}_{\alpha, \gamma}(v) \geq \frac{1}{7} \sqrt{\delta n}.
\end{align*}
\end{lemma}

To study the set of non-almost-constant vectors $\mathcal{N}(\delta, \rho)$, we partition it into level sets according to the magnitude of the CLCD. It suffices to show that, with high probability, the normal vector does not belong to any level set associated with a small CLCD.

The argument proceeds by constructing an approximating net of controlled cardinality. Exploiting the stability of the CLCD under small perturbations, we observe that if the normal vector has a small CLCD, then there exists a vector in the net with a comparably small CLCD. By applying small ball probability estimates to individual vectors in the net and taking a union bound, we obtain the desired probability bound.

Unless stated otherwise, we assume throughout this section that the parameters satisfy the hierarchy:
\begin{align}\label{eq:constants}
    0 < \delta, \rho \ll 1, \quad \text{and} \quad 0 < \mu \ll_{\delta, \rho} \gamma \ll_{\delta, \rho} 1.
\end{align}

Let $H_0 := \frac{1}{7}\sqrt{\delta n}$. By Lemma \ref{lemma:large CLCD}, we have the lower bound $\operatorname{CLCD}_{\alpha, \gamma}(v) \ge H_0$ for every $v \in \mathcal{N}(\delta, \rho)$. This motivates the following filtration of the sphere based on the CLCD.

\begin{mydef}[Level sets of CLCD]
Let $H \ge H_0/2$. We define the level set $S_H \subseteq \mathbb{S}^{n-1}$ associated with the parameter $H$ as
\begin{align*}
S_H := \left\{ v \in \mathcal{N}(\delta, \rho) : H \le \operatorname{CLCD}_{\mu n, \gamma}(v) \le 2H \right\}.
\end{align*}   
\end{mydef}

To bound the small ball probability, we recall the tensorization lemma of Rudelson and Vershynin \cite[Lemma 2.2]{MR2407948}.

\begin{lemma}[Tensorization] \label{lemma:tensorization}
Let $\varepsilon_0 \in (0, 1)$ and $B \ge 1$. Let $X_1, \dots, X_m$ be independent
 random variables such that each $X_i$ satisfies
\begin{align*}
\textbf{P}\left(|X_i| \le \varepsilon\right) \le B\varepsilon \quad \text{for all } \varepsilon \ge \varepsilon_0.
\end{align*}
Then there exists a universal constant $C > 0$ such that
\begin{align*}
\textbf{P}\left( \|(X_1, \dots, X_m)\|_2 \le \varepsilon \sqrt{m} \right) \le (CB\varepsilon)^m \quad \text{for every } \varepsilon \ge \varepsilon_0.
\end{align*}
\end{lemma}

The tensorization lemma tells us how to bound the anti-concentration of the random vector $M_n v$ for a fixed $v$. Specifically, let $R_1, \dots, R_n$ denote the independent rows of the matrix $M_n$. We observe that $\|M_n v\|_2^2 = \sum_{i=1}^n \langle R_i, v \rangle^2$.
We apply Lemma \ref{lemma:tensorization} to the random variables $X_i := \langle R_i, v \rangle$. 
In particular,  we have
\[
X_i=\langle R_i,v\rangle \stackrel{d}{=} W_v,
\]
where \(W_v\) is defined in \eqref{W_v}.
Moreover, we can use Theorem \ref{LOT} to bound the Lévy concentration function of each $X_i$. This yields the following lemma.

\begin{lemma}\label{lemma:Invertibility_on_a_single_vector_via_small_ball_probability}
For any $b > 0$ and $\mu, \gamma \in (0, 1)$, there exist $C_7,c_7  > 0$ such that the following hold. For any $v \in \mathbb{R}^n$ with $\|\mathrm{D}(v)\|_2 \geq b\sqrt{n}$ and any $\varepsilon \geq \frac{1}{\operatorname{CLCD}_{\mu n, \gamma}(v)} + e^{-c_7 n}$, we have
\begin{align*}
\textbf{P} \left( \|M_n v\|_2 \leq \varepsilon \sqrt{n} \right) \le (C_7 \varepsilon)^n.
\end{align*}
\end{lemma}

\begin{lemma}\label{lemma:discretization_level_sets}
Assume that the parameters $\delta, \rho, \mu$ and $\gamma$ satisfy \eqref{eq:constants}. Then there exists a net $\mathcal{F}\subset S_H + \frac{4\mu\sqrt n}{H}B_2^n$ of cardinality at most $\mu^{-2} H^2 \cdot (C_8 H/\sqrt{n})^n$ such that for any $v$ in $S_H$, there exists a vector $w$ in $\mathcal{F}$ such that $\|v-w\|_2\le4\mu\sqrt{n}/H$. 
\end{lemma}
\begin{proof}
We study the additive structure of the level set $S_H$ to construct an approximating net with controlled cardinality. Fix a vector $v = (v_1, \dots, v_n) \in S_H$ and let
\begin{align*}
    \sigma := \left\{ i \in [n] : |v_i| \le \sqrt{2/n} \right\}.
\end{align*}
Given the normalization $\|v\|_2 = 1$, it follows from a simple calculation that $|\sigma| \ge n/2$. Let $T := \operatorname{CLCD}_{\mu n, \gamma}(v)$. By the definition of the level set $S_H$, we have $H \le T < 2H$. 
According to the definition of the CLCD, there exists a lattice point $p = (p_{ij})_{1 \le i < j \le n} \in \mathbb{Z}^{\binom{n}{2}}$ such that 
\begin{align}\label{eq:distance1}
    \|T D(v) - p\|_2 < \mu n. 
\end{align}
For each $1 \le j \le n$, we define the auxiliary vectors $v^{(j)} \in \mathbb{R}^{n-1}$ and $p^{(j)} \in \mathbb{Z}^{n-1}$ as follows:
\begin{align*}
    v^{(j)} &:= (v_1 - v_j, \dots, v_{j-1} - v_j, v_{j} - v_{j+1}, \dots, v_j - v_n), \\
    p^{(j)} &:= (p_{1j}, \dots, p_{j-1\,j}, p_{j\,j+1}, \dots, p_{jn}).
\end{align*}
\eqref{eq:distance1} implies the following distance estimate:
\begin{align*}
    \sum_{j \in [n]} \|Tv^{(j)} - p^{(j)}\|_2^2 = 2\|T D(v) - p\|_2^2 < 2(\mu n)^2.
\end{align*}
Recalling that $|\sigma| \ge n/2$, a pigeonhole argument ensures the existence of an index $j \in \sigma$ such that 
\begin{align*}
    \|Tv^{(j)} - p^{(j)}\|_2^2 \le \frac{2(\mu n)^2}{|\sigma|} \le 4\mu^2 n.
\end{align*}
Normalizing by $T$ and utilizing the lower bound $T \ge H$, we obtain that for some $j \in [n]$:
\begin{align}\label{eq:distance2}
    \left\| v^{(j)} - \frac{p^{(j)}}{T} \right\|_2 \le \frac{2\mu \sqrt{n}}{T} \le \frac{2\mu \sqrt{n}}{H}. 
\end{align}
To control the magnitude of the lattice point $p^{(j)}$, we first observe that by the inequality $(x+y)^2 \le 2x^2 + 2y^2$,
\begin{align*}
    \|v^{(j)}\|_2^2 \le 2n v_j^2 + 2(v_1^2 + \dots + v_n^2) \le 6,
\end{align*}
where we have used the bound $|v_j| \le \sqrt{2/n}$ and the unit norm assumption $\|v\|_2 = 1$. Combining this bound with \eqref{eq:distance2} yields
\begin{align*}
    \|p^{(j)}\|_2 \le T\|v^{(j)}\|_2 + \|Tv^{(j)} - p^{(j)}\|_2 \le 2\sqrt{6}H  + 2\mu \sqrt{n} \le 7H,
\end{align*}
where the last inequality follows from $H \ge \frac{1}{14}\sqrt{\delta n}$ .

To localize the original vector $v$, we discretize the respective ranges of $v_j$ and $T$. Let $\Lambda$ and $\Theta$ be the discrete sets defined as follows:
\begin{align*}
    \Lambda := \frac{\mu}{2H} \mathbb{Z} \cap [-1, 1], \quad \Theta := \frac{1}{7} \mu \mathbb{Z} \cap [H, 2H].
\end{align*}
For any $j \in \sigma$ and $T \in [H, 2H]$, we can find $\lambda_0 \in \Lambda$ and $T_0 \in \Theta$ satisfying the approximation bounds,
\begin{align*}
    |v_j - \lambda_0| \le \frac{\mu}{2H}, \quad |T - T_0| \le \mu/7.
\end{align*}
Defining the vector $w = \left( \frac{p_{1j}}{T_0} + \lambda_0, \cdots, \frac{p_{j-1\,j}}{T_0} + \lambda_0, \lambda_0, -\frac{p_{j\,j+1}}{T_0} + \lambda_0, \cdots, -\frac{p_{jn}}{T_0} + \lambda_0 \right)$, then we obtain
\begin{align*}
\|v - w\|_2^2& = \sum_{i=1}^{j-1} \left( v_i - \frac{p_{ij}}{T_0} - \lambda_0 \right)^2+ (v_j - \lambda_0)^2 + \sum_{k=j+1}^{n} \left( v_k + \frac{p_{jk}}{T_0} - \lambda_0 \right)^2 \\
&\leq 3 \sum_{i=1}^{j-1} \left\{ \left( v_i - v_j - \frac{p_{ij}}{T} \right)^2 + (v_j - \lambda_0)^2 + \left( \frac{p_{ij}}{T} - \frac{p_{ij}}{T_0} \right)^2 \right\} + (v_j - \lambda_0)^2 \\
&\quad + 3 \sum_{k=j+1}^{n} \left\{ \left( v_k - v_j + \frac{p_{jk}}{T} \right)^2 + (v_j - \lambda_0)^2 + \left( \frac{p_{jk}}{T_0} - \frac{p_{jk}}{T} \right)^2 \right\} \\
&= 3\left\| v^{(j)} - \frac{p^{(j)}}{T} \right\|_2^2 + (3n - 2)(v_j - \lambda_0)^2 + \left( \frac{1}{T} - \frac{1}{T_0} \right)^2 \left\| p^{(j)} \right\|_2^2 \leq \frac{14\mu^2 n}{H^2},
\end{align*}
where the first inequality follows from $(x+y+z)^2 \le 3(x^2 + y^2 + z^2)$, while the second inequality follows from $\left\| v^{(j)} - \frac{p^{(j)}}{T} \right\|_2 \leq 2\mu\sqrt{n}/H$, $|v_j - \lambda_0| \leq \mu/(2H)$, $\left| \frac{1}{T} - \frac{1}{T_0} \right| \leq \mu/(7H^2)$ and $\left\| p^{(j)} \right\| \leq 7H$.

The above study tells us that $v$ is within the Euclidean distance $4\mu \sqrt{n}/H$ of the set $\mathcal{F}_j$ defined by
\begin{align*}
    \mathcal{F}_j := \left\{ \left( \frac{q_1}{T_0} + \lambda_0, \dots, \frac{q_n}{T_0} + \lambda_0 \right) : \lambda_0 \in \Lambda, T_0 \in \Theta, q \in \mathbb{Z}^{n}, q_j=0, \|q\|_2\le7H \right\}.
\end{align*}

There are at most $1 + 4H/\mu$ choices for $\lambda_0 \in \Lambda$, at most $1 + 7H/\mu$ ways to choose $T_0 \in \Theta$. Furthermore, the number of lattice points $q$ in the $n-1$-dimensional ball $B(0, 7H)$ is bounded by $(1 + cH/\sqrt{n})^n$ .

Define $\mathcal F$ to be the set of all vectors $w=w(v)$ produced by the above rounding procedure from some $v\in S_H$. Then
\begin{align*}
|\mathcal{F}|\le|\mathcal{F}_1\cup\cdots\cup\mathcal{F}_n| \leq n\cdot(1 + 4H/\mu) \cdot (1 + 7H/\mu) \cdot (1 + cH/\sqrt{n})^n \leq \mu^{-2} H^2 (C_8 H/\sqrt{n})^n,
\end{align*}
where the last inequality follows from the lower bound \(H\ge H_0=\frac17\sqrt{\delta n}\). This completes our proof.
\end{proof}

To conveniently study the invertibility of vectors in $\mathcal{F}$, we use the stability of CLCD (Lemma \ref{lemma:clcd_stability}) to study the CLCD of vectors in $\mathcal{F}$.
\begin{lemma}[CLCD on the net]
Let $\mathcal{F}$ be the net defined in Lemma \ref{lemma:discretization_level_sets} and assume that the parameters $\delta, \rho, \mu$ and $\gamma$ satisfy \eqref{eq:constants}.  For every $w \in \mathcal{F}$, we have $\operatorname{CLCD}_{\mu n/2,\gamma/2}(w) \geq H/16$ and $\|D(w) \|_2\gtrsim_{\delta,\rho} \sqrt{n}$.
\end{lemma}
\begin{proof}
Since net $\mathcal{F}\subset S_H + \frac{4\mu\sqrt n}{H}B_2^n$, for any vector $w$ in the net $\mathcal{F}$, there exists a vector $v \in S_H$ such that $\|v - w\|_2 \le 4\mu\sqrt{n}/H$. Since $\mu \ll_{\delta} 1$, we observe that $\|v - w\|_2 \lesssim_\delta \mu$.

Recall that for $v \in S_H \subseteq \mathcal{N}(\delta, \rho)$, Lemma \ref{lemma:D(v) lower bound} implies the lower bound $\|D(v)\|_2 \gtrsim_{\delta, \rho} \sqrt{n}$. Given the condition $\mu \ll_{\delta, \rho} \gamma$, thus $\|v - w\|_2\lesssim_\delta \mu < \frac{\gamma \|D(v)\|_2}{5\sqrt{n}}$ .
Therefore by Lemma \ref{lemma:clcd_stability}, we obtain
\begin{align*}
    \operatorname{CLCD}_{\mu n/2, \gamma/2}(w) \ge \min \left\{ \operatorname{CLCD}_{\mu n, \gamma}(v), \frac{\mu \sqrt{n}}{4 \|v - w\|_2} \right\} \ge H/16 ,
\end{align*}
where the last inequality follows from the facts that $\operatorname{CLCD}_{\mu n, \gamma}(v) \ge H$ and $\|v - w\|_2 \le 4\mu\sqrt{n}/H$. 

Furthermore, since for any $x\in\mathbb{R}^n$ we have $\|D(x)\|_2\le\sqrt{n}\|x\|_2$, by the triangle inequality, we obtain
\begin{align*}
    \|D(w)\|_2 \ge \|D(v)\|_2 - \|D(v - w)\|_2 \ge \|D(v)\|_2 - \sqrt{n} \|v - w\|_2
    &\ge \|D(v)\|_2 / 2 \gtrsim_{\delta, \rho} \sqrt{n} .
\end{align*}
In the last inequality, we use the bound $\|v - w\|_2 \le \frac{\gamma \|D(v)\|_2}{5\sqrt{n}} \le \frac{\|D(v)\|_2}{5\sqrt{n}}$. This completes the proof.
\end{proof}

\begin{lemma}\label{lemma:Invertibility_on_a_level_set}
There exist constants $\mu, \gamma, c_9 \in (0, 1)$ and $C_9 > 0$ such that the following holds. Suppose that $n \ge C_9$ and $H_0 \le H \le e^{c_9 n}$. Then
\begin{align*}
\textbf{P} \left( \inf_{v \in S_H} \|M_n' v\|_2 \le \frac{c_9 n}{H} \right) \le 2e^{-n}.
\end{align*}
\end{lemma}

\begin{proof}
By Proposition \ref{prop:first_singular_value}, there exists a constant $K \ge 1$ such that the event $\{ \|M_n'\|  > K\sqrt{n} \}$ occurs with probability at most $e^{-n}$. Consequently, to conclude the proof, it suffices to establish the existence of constants $C, c > 0$ such that for all $n \ge C$ and $H_0 \le H \le e^{c n}$, the event
\begin{align*}
    \mathcal{E} := \left\{ \inf_{v \in S_H} \|M_n' v\|_2 \le \frac{cn}{2H} \quad \text{and} \quad  \|M_n'\|  \le K\sqrt{n} \right\}
\end{align*}
has probability bounded by $e^{-n}$.

We set the parameter choices as follows:
\begin{align*}
    c = \min \left\{ \frac{1}{e^5C_7 C_8}, c_7, 1 \right\}, \quad C = \max \{ (32/c)^2, (C_7 c)^{-2} \},
\end{align*}
where $C_7, c_7$ are the constants associated with the small ball probability estimates (Lemma \ref{lemma:Invertibility_on_a_single_vector_via_small_ball_probability}), and $C_8$ is the constant from the net construction (Lemma \ref{lemma:discretization_level_sets}). Furthermore, we set the parameters $0 < \delta, \gamma \ll 1$ and $0 < \mu \ll_{\delta, \rho, K} \gamma$.

Let $\mathcal{F}$ be the net constructed in Lemma \ref{lemma:discretization_level_sets}. For any $w \in \mathcal{F}$, as established in the proof of Lemma \ref{lemma:discretization_level_sets}, we have $\operatorname{CLCD}_{\mu n/2, \gamma/2}(w) \ge H/16$ and $\|D(w)\|_2 \gtrsim_{\delta, \rho} \sqrt{n}$. 

Define the threshold $\varepsilon := c\sqrt{n}/H$. Under the assumption $n \ge (32/c)^2$ and $c\le c_7$, this threshold satisfies $\varepsilon \ge \frac{1}{\operatorname{CLCD}_{\mu n/2, \gamma/2}(w)} + e^{-c_7 n}.$ Thus, similarly to Lemma \ref{lemma:Invertibility_on_a_single_vector_via_small_ball_probability}, we obtain
\begin{align*}
    \textbf{P} \left(\|M_n' w\|_2 \le \frac{cn}{H} \right) \le \left( \frac{C_7 c \sqrt{n}}{H} \right)^{n-1}.
\end{align*}
We now apply a union bound over the net $\mathcal{F}$. Recalling the cardinality bound $|\mathcal{F}| \le \mu^{-2} H^2 (C_8 H / \sqrt{n})^n$ and $n \ge (C_7 c)^{-2}$ and $H \le e^{cn}\le e^n$, we have
\begin{align*}
\textbf{P} \left( \inf_{w \in \mathcal{F}} \|M_n' w\|_2 \le \frac{cn}{H} \right)
&\le \mu^{-2} H^2 \left( \frac{C_8 H}{\sqrt{n}} \right)^n \left( \frac{C_7 c \sqrt{n}}{H} \right)^{n-1}  \\
&\le \mu^{-2} H^3 \cdot (C_8 C_7 c)^n  \\
&\le \mu^{-2}e^{3n} \cdot e^{-5n} \le e^{-n}.
\end{align*}
Finally, assume the event $\mathcal{E}$ occurs. Then there exists $v \in S_H$ with $\|M_n' v\|_2 \le cn/(2H)$. By the approximation property of the net, there exists $w \in \mathcal{F}$ such that $\|v - w\|_2 \le 4\mu\sqrt{n}/H$. Then, we obtain
\begin{align*}
    \|M_n' w\|_2 \le \|M_n' v\|_2 + \|M_n'\|  \|v - w\|_2 \le \frac{cn}{2H} + K\sqrt{n} \cdot \frac{4\mu\sqrt{n}}{H} \le \frac{cn}{H},
\end{align*}
where $\mu$ is chosen sufficiently small ($8K\mu \le c$). This implies that there exists a constant $c_9\in(0,1)$ such that
\begin{align*}
\textbf{P} \left( \mathcal{E} \right) \le \textbf{P} \left( \inf_{w \in \mathcal{F}} \|M_n' w\|_2 \le \frac{c_9n}{H} \right)\le e^{-n}.
\end{align*}
This completes the proof.
\end{proof}

In the remainder of this section, we prove Proposition \ref{prop:random_normal} and Theorem \ref{theorem:distance}.
\begin{proof}[Proof of Proposition \ref{prop:random_normal}]
Let $\mu, \gamma, c_9 \in (0, 1)$ be the constants established in Lemma \ref{lemma:Invertibility_on_a_level_set}. We focus on the event
\begin{align*}
    \mathcal{E} := \left\{ \exists v \in \mathcal{N}(\delta, \rho) : M_n' v = 0  ,\ \operatorname{CLCD}_{\mu n, \gamma}(v) \le e^{c_9 n} \right\}.
\end{align*}
According to Lemma \ref{lemma:large CLCD}, for all $v\in\mathcal{N}(\delta, \rho)$, we have $\operatorname{CLCD}_{\mu n, \gamma}(v)\ge H_0 = \frac{1}{7}\sqrt{\delta n}$. Thus, we have
\begin{align*}
\textbf{P}(\mathcal{E}) \le \sum_{k} \textbf{P} \left\{ \exists v \in S_{2^k} : M_n' v = 0 \right\},
\end{align*}
where the sum runs over all integers $k$ satisfying $H_0/2 \le 2^k \le e^{c_9 n}$.

By Lemma \ref{lemma:Invertibility_on_a_level_set}, for any $H = 2^k$, we have
\begin{align*}
    \textbf{P} \left\{ \exists v \in S_{2^k} : M_n' v = 0 \right\} \le \textbf{P} \left( \inf_{v \in S_{2^k}} \|M_n' v\|_2 \le \frac{c_9 n}{2^k} \right) \le 2e^{-n},
\end{align*}
for sufficiently large $n$. Applying the union bound yields
\begin{align*}
    \textbf{P}(\mathcal{E}) \le (c_9 n \log_2 e) \cdot 2e^{-n}\le4c_9n\cdot e^{-n}\le2^{-n}.
\end{align*}
Let $c_6=c_9$, this completes the proof.
\end{proof}
\begin{proof}[Proof of Theorem \ref{theorem:distance}]
Choose parameters $\delta$ and $\rho$ such that $0 < \delta, \rho \ll 1$. It follows from Proposition \ref{prop:Random normal is non-almost-constant} that with probability at least $1 - e^{-c_5n}$ any unit vector orthogonal to $H_n$ is in $\mathcal{N}(\delta, \rho)$. On the other hand, Proposition \ref{prop:random_normal} implies that, with probability at least \(1-2^{-n}\), every unit vector \(v\in H_n^\perp\cap \mathcal N(\delta,\rho)\) satisfies
\begin{align*}
\operatorname{CLCD}_{\mu n,\gamma}(v)\ge e^{c_6n}.
\end{align*}
Combining these two events and applying the union bound, we conclude that, with probability at least \(1-e^{-c_5n}-2^{-n}\), any unit vector \(v\in H_n^\perp\) satisfies
\begin{align*}
v\in \mathcal N(\delta,\rho)
\quad\text{and}\quad
\operatorname{CLCD}_{\mu n,\gamma}(v)\ge e^{c_6n}.
\end{align*}
Applying Theorem \ref{LOT} to such a vector \(v\) completes the proof.

\end{proof}

\section{Proof of Theorem \ref{main}}\label{sec5}
Consider the event
\begin{align*}
\mathcal{E} := \left\{ \exists v \in \mathbb{S}^{n-1} \text{ such that } \|v^{\top}M_n \|_2 \leq \varepsilon \frac{\rho}{\sqrt{n}} \right\}.
\end{align*}
\noindent Recalling the decomposition of the sphere $\mathbb{S}^{n-1} = \operatorname{Comp}(\delta, \rho) \cup \operatorname{Incomp}(\delta, \rho)$, we break the invertibility problem into two subproblems. For compressible vectors, we have the following proposition
\begin{proposition}\label{prop:Invertibility for compressible vectors}
Let $M_n$ be a matrix as in Theorem \ref{main}. Then there exist constants $\delta, \rho, c_{10} \in(0,1)$ such that
\begin{align*}
\textbf{P}\left(\inf _{x \in \operatorname{Comp}(\delta, \rho)}\left\|x^{\top} M_n\right\|_2 \leq \sqrt{n} / 144\right) \leq 2 e^{-c_{10} n}.
\end{align*}
\end{proposition}
Since we consider a vector being left-multiplied by a matrix, we need the invertibility of a single left-multiplied vector.
\begin{lemma}\label{LV}
Let $M_n$ be as in Theorem \ref{main}. Then, for any vector $x$ in $\mathbb{S}^{n-1}$, we have
\begin{align*}
\textbf{P}\left(\left\|x^{\top} M_n\right\|_2 \leq \sqrt{n} / 36\right) \leq e^{-{n}/{400}}.
\end{align*}
\end{lemma}
We first recall the Paley–Zygmund inequality \cite[Lemma 3.5]{MR2146352}.
\begin{lemma}\label{PZI}
Let $X$ be a nonnegative random variable with $\textbf{E}[X^4] < \infty$. Then, for $0 \leq \lambda < \sqrt{\textbf{E}[X^2]}$ we have
\begin{align*}
\textbf{P}\left(X > \lambda\right) \geq \frac{(\textbf{E}[X^2] - \lambda^2)^2}{\textbf{E}[X^4]}.
\end{align*}
\end{lemma}

\begin{proof}
Let $N:=n/4$, and let $\Gamma$ denote the $n\times N$ submatrix of $M_n$ formed by the first $N$ columns.
For $j=1,\dots,N$, let $\Gamma_j$ be the $j$-th column of $\Gamma$ and $\Gamma_{ij}$ be the entries of $\Gamma$, then set $y_j := (x^\top \Gamma)_j = x^\top \Gamma_j = \sum_{i=1}^n x_i \Gamma_{ij}.$
Then
\begin{align*}
\|x^\top M_n\|_2\ge \|x^\top \Gamma\|_2
= \left(\sum_{j=1}^N y_j^2\right)^{1/2}.
\end{align*}
Therefore, for any $t>0$,
\begin{align*}
\textbf{P} \left(\|x^\top M_n\|_2 \le t\sqrt{N}\right)\le\textbf{P}\left(\sum_{j=1}^N y_j^2 \le t^2 N\right).
\end{align*}
For every $\tau>0$, by the exponential Markov inequality,
\begin{align*}
\textbf{P}\left(\sum_{j=1}^N y_j^2 \le t^2 N\right)
&= \textbf{P}\left(-\frac{\tau}{t^2}\sum_{j=1}^N y_j^2
\ge -\tau N\right) \le e^{\tau N}\,
\textbf{E}\exp\left(-\frac{\tau}{t^2}\sum_{j=1}^N y_j^2\right).
\end{align*}
Iterating conditional expectations yields
\begin{align*}
\textbf{E}\exp\left(-\frac{\tau}{t^2}\sum_{j=1}^N y_j^2\right) \le \prod_{j=1}^N \sup_{\Gamma_1,\dots,\Gamma_{j-1}}
\textbf{E}\left(e^{-\tau y_j^2/t^2}\mid\,\Gamma_1,\dots,\Gamma_{j-1}\right),
\end{align*}
where the supremum is taken over all realizations of the first $j-1$ columns.
Hence
\begin{align}\label{eq2440}
\textbf{P}\left(\sum_{j=1}^N y_j^2 \le t^2 N\right)\le e^{\tau N}\prod_{j=1}^N \sup_{\Gamma_1,\dots,\Gamma_{j-1}}\textbf{E}\left(e^{-\tau y_j^2/t^2}\mid\,\Gamma_1,\dots,\Gamma_{j-1}\right).
\end{align}

Now, we study $y_j$ conditional on $\Gamma_1,\dots,\Gamma_{j-1}$. By construction of the model, the $i$-th row must have exactly $n/2$ non-zero entries, chosen uniformly and without replacement.
Given $\Gamma_1,\dots,\Gamma_{j-1}$, define $p_i:=\textbf{P}\left(\Gamma_{ij}\ne 0\mid\Gamma_1,\dots,\Gamma_{j-1}\right).$ Since $j\le N=n/4$, we have $\frac13 \le p_i \le \frac23 $ for all $i\le n$.

Conditioned on the event $\Gamma_{ij}\ne 0$, $\Gamma_{ij}$ is a Rademacher random variable. Therefore,
\begin{align*}
\textbf{P}(\Gamma_{ij}=1\mid\Gamma_1,\dots,\Gamma_{j-1}) = \textbf{P}(\Gamma_{ij}=-1\mid\Gamma_1,\dots,\Gamma_{j-1}) = \frac{p_i}{2},
\quad
\textbf{P}(\Gamma_{ij}=0\mid\Gamma_1,\dots,\Gamma_{j-1}) = 1-p_i.
\end{align*}

Define $Z_{ij} := x_i \Gamma_{ij}$ so that $y_j = \sum_{i=1}^n Z_{ij}$. Given $\Gamma_1,\dots,\Gamma_{j-1}$, the random variables $\{Z_{ij}\}_{i=1}^n$ are independent and centered. Since $\textbf{E}[Z_{ij}\mid\Gamma_1,\dots,\Gamma_{j-1}]=0$ and $Z_{ij}$ are independent,
\begin{align*}
\textbf{E}[y_j^2\mid\Gamma_1,\dots,\Gamma_{j-1}]
=\sum_{i=1}^n \textbf{E}[Z_{ij}^2\mid\Gamma_1,\dots,\Gamma_{j-1}]
=\sum_{i=1}^n x_i^2\,\textbf{E}[\Gamma_{ij}^2\mid\Gamma_1,\dots,\Gamma_{j-1}]
=\sum_{i=1}^n x_i^2 p_i.
\end{align*}
Using $p_i\in[1/3,2/3]$ and $\sum_{i=1}^n x_i^2=1$ we obtain
\begin{align}\label{eq2241}
\frac13 \le \textbf{E}[y_j^2\mid\Gamma_1,\dots,\Gamma_{j-1}].
\end{align}
Similarly, since
$\textbf{E}[Z_{ij}^2\mid\Gamma_1,\dots,\Gamma_{j-1}]=x_i^2 p_i$ and
$\textbf{E}[Z_{ij}^4\mid\Gamma_1,\dots,\Gamma_{j-1}]=x_i^4 \textbf{E}[\Gamma_{ij}^4\mid\Gamma_1,\dots,\Gamma_{j-1}]
= x_i^4 p_i,$
\begin{align*}
\textbf{E}[y_j^4\mid\Gamma_1,\dots,\Gamma_{j-1}]
&= \sum_{i=1}^n \textbf{E}[Z_{ij}^4\mid\Gamma_1,\dots,\Gamma_{j-1}]
+ 6\sum_{1\le i<k\le n} \textbf{E}[Z_{ij}^2\mid\Gamma_1,\dots,\Gamma_{j-1}]
\textbf{E}[Z_{kj}^2\mid\Gamma_1,\dots,\Gamma_{j-1}]\\
&=\sum_{i=1}^n x_i^4 p_i
+ 6\sum_{1\le i<k\le n} x_i^2 x_k^2 p_i p_k\le3\left[\textbf{E}[y_j^2\mid\Gamma_1,\dots,\Gamma_{j-1}]\right]^2.
\end{align*}

We now set $\lambda=\sqrt{\textbf{E}[y_j^2\mid\Gamma_1,\dots,\Gamma_{j-1}]/2}$ and apply the conditional version of the Paley-Zygmund inequality (Lemma \ref{PZI}) conditionally on $\Gamma_1,\dots,\Gamma_{j-1}$. Since \eqref{eq2241}, we have $\lambda\ge1/3$. Then, we obtain
\begin{align*}
\textbf{P}\left(|y_j|>1/3 \mid\Gamma_1,\dots,\Gamma_{j-1}\right)\ge
\frac{\left(\textbf{E}[y_j^2\mid\Gamma_1,\dots,\Gamma_{j-1}]-\lambda^2\right)^2}
{\textbf{E}[y_j^4\mid\Gamma_1,\dots,\Gamma_{j-1}]}
= \frac{(\textbf{E}[y_j^2\mid\Gamma_1,\dots,\Gamma_{j-1}]/2)^2}{\textbf{E}[y_j^4\mid\Gamma_1,\dots,\Gamma_{j-1}]}\ge \frac{1}{12}.
\end{align*}
Consequently,
\begin{align*}
\textbf{E}\left(e^{-\tau y_j^2/t^2}\,\Big|\,\Gamma_1,\dots,\Gamma_{j-1}\right)
&= \textbf{E}\left(e^{-\tau y_j^2/t^2}\mathbf 1_{\{|y_j|\le 1/3\}}\,\Big|\,\Gamma_1,\dots,\Gamma_{j-1}\right)
+ \textbf{E}\left(e^{-\tau y_j^2/t^2}\mathbf 1_{\{|y_j|>1/3\}}\,\Big|\,\Gamma_1,\dots,\Gamma_{j-1}\right)\\
&\le \textbf{P}\left(|y_j|\le 1/3\mid\Gamma_1,\dots,\Gamma_{j-1}\right)
+ e^{-\tau/(9t^2)}\,\textbf{P}\left(|y_j|>1/3\mid\Gamma_1,\dots,\Gamma_{j-1}\right)\\
&\le \frac{11}{12} + \frac{1}{12}\,e^{-\tau/(9t^2)}.
\end{align*}
Plugging this into \eqref{eq2440}, we obtain
\begin{align*}
\textbf{P}\left(\sum_{j=1}^N y_j^2 \le t^2 N\right)\le\exp(\tau N)\left(\frac{11}{12} + \frac{1}{12}\,e^{-\tau/(9t^2)}\right)^N.
\end{align*}

We now choose $t=1/18$ and $\tau=1/100$. Since $N=n/4$, we obtain
\begin{align*}
\textbf{P}\left(\|x^\top M_n\|_2 \le \sqrt n/36\right)
\le \textbf{P}\left(\sum_{j=1}^N y_j^2 \le t^2 N\right)
\le e^{-{n}/{400}},
\end{align*}
for every $x\in \mathbb{S}^{n-1}$. This completes the proof.
\end{proof}

Next, it suffices to construct a suitable net to complete the proof of Proposition \ref{prop:Invertibility for compressible vectors}. Recall a special case of Theorem 4 from \cite{MR4361906}:
\begin{lemma}[{\cite[Theorem 4]{MR4361906}}]\label{SNDM}
Consider any set  $S \subseteq \mathbb{S}^{n-1}$ , and fix any $ \alpha \in (0, \frac{1}{2}) $, $ \beta \in (0, \frac{\alpha}{10}) $. Assume $ n \geq 1/\alpha^2 $. Then there exists a (deterministic) net $ \mathcal{N} \subseteq S + \frac{4\beta}{\alpha} B_2^n $ such that
\begin{align*}
|\mathcal{N}| \leq N(S, \beta B_2^n) \cdot e^{C_{11} \alpha^{0.08} \log(1/\alpha) n},
\end{align*}
where $ C_{11} > 0 $ is an absolute constant. Moreover, for any (deterministic) $ n \times n $ matrix $ A $, the following holds: for every $ x \in S $, there exists $ y \in \mathcal{N} $ such that
\begin{align*}
\|(x - y)^{\top} A\|_2 \leq \frac{2\beta}{\alpha \sqrt{n}} \|A\|_{\mathrm{HS}}.
\end{align*}
\end{lemma}

\begin{proof}[Proof of Proposition \ref{prop:Invertibility for compressible vectors}]
Set $\alpha=288(\delta+2\rho)$ and $\beta = \delta+2\rho$, where $0<\delta,\rho\ll1$. Assume that
$n\ge1/\alpha^2$. Observe that
\begin{align*}
N\left(\operatorname{Comp}(\delta,\rho),\beta B_2^n\right)\le N\left(\mathrm{Sparse}(\delta)\cap\mathbb{S}^{n-1},(\beta-2\rho)B_2^n\right)\le\binom{n}{\delta n}\left(\frac{4}{\delta}\right)^{\delta n}
\le e^{2\delta\log(4/\delta)\,n}.
\end{align*}
Using Lemma \ref{SNDM}, we get a net $\mathcal{N}\subseteq\frac32B_2^n\backslash\frac12B_2^n$ with
\begin{align*}
|\mathcal{N}| \leq e^{2\delta\log(4/\delta)\,n} \cdot e^{C_{11} \alpha^{0.08} \log(1/\alpha) n}\le e^{C \alpha^{0.08} \log(1/\alpha) n}.
\end{align*}

Consider the event
\begin{align*}
\mathcal{E}:=\left\{ \inf_{x\in \operatorname{Comp}(\delta,\rho)} \|x^\top M_n\|_2 \le \sqrt{n}/144 \right\}.
\end{align*}
We suppose that $\mathcal{E}$ occurs. Then
$\|x^\top M_n\|_2\le \sqrt{n}/144$ for some $x\in \operatorname{Comp}(\delta,\rho)$. Since
$\|M_n\|_{\mathrm{HS}}\le n$, Lemma \ref{SNDM} shows the existence of $y\in\mathcal{N}$ with
\begin{align*}
\|(x-y)^\top M_n\|_2\le (2\beta/\alpha)\sqrt{n}\le \sqrt{n}/144.
\end{align*}
In particular, we have
\begin{align*}
\|y^\top M_n\|_2
\le \|x^\top M_n\|_2 + \|(x-y)^\top M_n\|_2
\le \sqrt{n}/72.
\end{align*}
For any $y \in \mathcal{N}$, Lemma \ref{LV} implies that
\begin{align*}
\textbf{P}\left(\left\|y^{\top} M_n\right\|_2 \leq \sqrt{n} / 72\right) \leq e^{-n/400}.
\end{align*}
Taking the union bound,
\begin{align*}
\textbf{P}\left( \inf_{y\in\mathcal{N}} \|y^\top M_n\|_2 \le \sqrt{n}/72 \right)
\le e^{C\alpha^{0.08}\log(1/\alpha)n}\cdot e^{-n/400}
< e^{-n/500},
\end{align*}
for sufficiently small \(\alpha\). Therefore, we obtain
\begin{align*}
\textbf{P}(\mathcal{E})
\le \textbf{P}\left( \inf_{y\in\mathcal{N}} \|y^\top M_n\|_2 \le \sqrt{n}/72 \right)
\le e^{-n/500}.
\end{align*}
for sufficiently large $n$, which completes our proof.
\end{proof}
For the incompressible vectors, we use the following “invertibility via distance” lemma from \cite{MR2407948}:

\begin{lemma}[Invertibility via distance]
Let $M$ be any random matrix. Let $R_1, \ldots, R_n$ denote the row vectors of $M$, and let $H_k$ denote the span of all row vectors except the $k$-th. Then for every $\delta, \rho \in(0,1)$ and every $\varepsilon \geq 0$, one has
\begin{align*}
\textbf{P}\left(\inf _{x \in \operatorname{Incomp}(\delta, \rho)}\left\|x^{\top}M\right\|_2 \leq \varepsilon \frac{\rho}{\sqrt{n}}\right) \leq \frac{1}{\delta n} \sum_{k=1}^n \textbf{P}\left(\operatorname{dist}\left(R_k, H_k\right) \leq \varepsilon\right) .
\end{align*}
\end{lemma}
Since the rows of \(M_n\) are independent and identically distributed, we have
\begin{align*}
\textbf{P}\left(\inf _{x \in \operatorname{Incomp}(\delta, \rho)}\left\|x^{\top}M_n\right\|_2 \leq \varepsilon \frac{\rho}{\sqrt{n}}\right) \leq \frac{1}{\delta n} \sum_{k=1}^n \textbf{P}\left(\operatorname{dist}\left(R_k, H_k\right) \leq \varepsilon\right)\le\frac{1}{\delta}  \textbf{P}\left(\operatorname{dist}\left(R_n, H_n\right) \leq \varepsilon\right) .
\end{align*}
Now, the proof of the main result is completed by an application of Theorem \ref{theorem:distance}.

\begin{proof}[Proof of Theorem \ref{main}]
\begin{align*}
\textbf{P}\left(s_n(M_n)\le\varepsilon\frac{ \rho}{\sqrt{n}}\right)&=\textbf{P}\left(\inf_{x\in\mathbb{S}^{n-1}}\|x^{\top}M_n\|_2\le\varepsilon\frac{ \rho}{\sqrt{n}}\right)\\
&\le\textbf{P}\left(\inf _{x \in \operatorname{Comp}(\delta, \rho)}\|x^{\top} M_n\|_2 \leq \varepsilon \frac{\rho}{\sqrt{n}}\right)+\textbf{P}\left(\inf _{x \in \operatorname{Incomp}(\delta, \rho)}\|x^{\top}M_n\|_2 \leq \varepsilon \frac{\rho}{\sqrt{n}}\right)\\
& \le\textbf{P}\left(\inf _{x \in \operatorname{Comp}(\delta, \rho)}\|x^{\top}M_n\|_2 \leq \sqrt{n} / 144\right)+\frac{1}{\delta} \textbf{P}\left(\operatorname{dist}\left(R_n, H_n\right) \leq \varepsilon\right)\\
& \leq 2 e^{-c_{10} n}+\frac{1}{\delta}\left(C \varepsilon+2 e^{-c n}\right),
\end{align*}
for sufficiently large $n$, where the last inequality follows from Proposition \ref{prop:Invertibility for compressible vectors} and Theorem \ref{theorem:distance}.

\end{proof}

\printbibliography

@article {MR950344,
    AUTHOR = {Yin, Y. Q. and Bai, Z. D. and Krishnaiah, P. R.},
     TITLE = {On the limit of the largest eigenvalue of the
              large-dimensional sample covariance matrix},
   JOURNAL = {Probab. Theory Related Fields},
  FJOURNAL = {Probability Theory and Related Fields},
    VOLUME = {78},
      YEAR = {1988},
    NUMBER = {4},
     PAGES = {509--521},
      ISSN = {0178-8051,1432-2064},
       DOI = {10.1007/BF00353874},
       URL = {https://doi.org/10.1007/BF00353874},
}

@book {MR157875,
    AUTHOR = {von Neumann, John},
     TITLE = {Collected works. {V}ol. {V}: {D}esign of computers, theory of
              automata and numerical analysis},
      NOTE = {General editor: A. H. Taub},
 PUBLISHER = {The Macmillan Company, New York},
      YEAR = {1963},
     PAGES = {ix+784 pp. (1 plate)},
}

@article{doi:10.1137/0609045,
author = {Edelman, Alan},
title = {Eigenvalues and Condition Numbers of Random Matrices},
journal = {SIAM Journal on Matrix Analysis and Applications},
volume = {9},
number = {4},
pages = {543-560},
year = {1988},
doi = {10.1137/0609045},
URL = {https://doi.org/10.1137/0609045},
eprint = {https://doi.org/10.1137/0609045},
}

@article {MR2480613,
    AUTHOR = {Tao, Terence and Vu, Van H.},
     TITLE = {Inverse {L}ittlewood-{O}fford theorems and the condition
              number of random discrete matrices},
   JOURNAL = {Ann. of Math. (2)},
  FJOURNAL = {Annals of Mathematics. Second Series},
    VOLUME = {169},
      YEAR = {2009},
    NUMBER = {2},
     PAGES = {595--632},
      ISSN = {0003-486X,1939-8980},
       DOI = {10.4007/annals.2009.169.595},
       URL = {https://doi.org/10.4007/annals.2009.169.595},
}

@article {MR2434885,
    AUTHOR = {Rudelson, Mark},
     TITLE = {Invertibility of random matrices: norm of the inverse},
   JOURNAL = {Ann. of Math. (2)},
  FJOURNAL = {Annals of Mathematics. Second Series},
    VOLUME = {168},
      YEAR = {2008},
    NUMBER = {2},
     PAGES = {575--600},
      ISSN = {0003-486X,1939-8980},
       DOI = {10.4007/annals.2008.168.575},
       URL = {https://doi.org/10.4007/annals.2008.168.575},
}

@article {MR2407948,
    AUTHOR = {Rudelson, Mark and Vershynin, Roman},
     TITLE = {The {L}ittlewood-{O}fford problem and invertibility of random
              matrices},
   JOURNAL = {Adv. Math.},
  FJOURNAL = {Advances in Mathematics},
    VOLUME = {218},
      YEAR = {2008},
    NUMBER = {2},
     PAGES = {600--633},
      ISSN = {0001-8708,1090-2082},
   MRCLASS = {60E15 (60B20)},
  MRNUMBER = {2407948},
MRREVIEWER = {Ben\ Joseph\ Green},
       DOI = {10.1016/j.aim.2008.01.010},
       URL = {https://doi.org/10.1016/j.aim.2008.01.010},
}

@article {MR4076632,
    AUTHOR = {Tikhomirov, Konstantin},
     TITLE = {Singularity of random {B}ernoulli matrices},
   JOURNAL = {Ann. of Math. (2)},
  FJOURNAL = {Annals of Mathematics. Second Series},
    VOLUME = {191},
      YEAR = {2020},
    NUMBER = {2},
     PAGES = {593--634},
      ISSN = {0003-486X,1939-8980},
       DOI = {10.4007/annals.2020.191.2.6},
       URL = {https://doi.org/10.4007/annals.2020.191.2.6},
}

@article {MR4810062,
    AUTHOR = {Campos, Marcelo and Jenssen, Matthew and Michelen, Marcus and
              Sahasrabudhe, Julian},
     TITLE = {The singularity probability of a random symmetric matrix is
              exponentially small},
   JOURNAL = {J. Amer. Math. Soc.},
  FJOURNAL = {Journal of the American Mathematical Society},
    VOLUME = {38},
      YEAR = {2025},
    NUMBER = {1},
     PAGES = {179--224},
      ISSN = {0894-0347,1088-6834},
       DOI = {10.1090/jams/1042},
       URL = {https://doi.org/10.1090/jams/1042},
}

@article {MR4523249,
    AUTHOR = {Jain, Vishesh and Sah, Ashwin and Sawhney, Mehtaab},
     TITLE = {The smallest singular value of dense random regular digraphs},
   JOURNAL = {Int. Math. Res. Not. IMRN},
  FJOURNAL = {International Mathematics Research Notices. IMRN},
      YEAR = {2022},
    NUMBER = {24},
     PAGES = {19300--19334},
      ISSN = {1073-7928,1687-0247},
       DOI = {10.1093/imrn/rnab247},
       URL = {https://doi.org/10.1093/imrn/rnab247},
}

@misc{tran2020smallestsingularvaluerandom,
      title={The smallest singular value of random combinatorial matrices}, 
      author={Tuan Tran},
      year={2020},
      eprint={2007.06318},
      archivePrefix={arXiv},
      primaryClass={math.PR},
      url={https://arxiv.org/abs/2007.06318}, 
}

@article {MR3602844,
    AUTHOR = {Cook, Nicholas A.},
     TITLE = {On the singularity of adjacency matrices for random regular
              digraphs},
   JOURNAL = {Probab. Theory Related Fields},
  FJOURNAL = {Probability Theory and Related Fields},
    VOLUME = {167},
      YEAR = {2017},
    NUMBER = {1-2},
     PAGES = {143--200},
      ISSN = {0178-8051,1432-2064},
       DOI = {10.1007/s00440-015-0679-8},
       URL = {https://doi.org/10.1007/s00440-015-0679-8},
}

@book {MR3837109,
    AUTHOR = {Vershynin, Roman},
     TITLE = {High-dimensional probability},
    SERIES = {Cambridge Series in Statistical and Probabilistic Mathematics},
    VOLUME = {47},
      NOTE = {An introduction with applications in data science,
              With a foreword by Sara van de Geer},
 PUBLISHER = {Cambridge University Press, Cambridge},
      YEAR = {2018},
     PAGES = {xiv+284},
      ISBN = {978-1-108-41519-4},
       DOI = {10.1017/9781108231596},
       URL = {https://doi.org/10.1017/9781108231596},
}

@book{milman1986asymptotic,
  title={Asymptotic theory of finite dimensional normed spaces},
  author={Milman, Vitali D and Schechtman, Gideon},
  year={1986},
  publisher={Springer}
}

@article {MR2146352,
    AUTHOR = {Litvak, A. E. and Pajor, A. and Rudelson, M. and
              Tomczak-Jaegermann, N.},
     TITLE = {Smallest singular value of random matrices and geometry of
              random polytopes},
   JOURNAL = {Adv. Math.},
  FJOURNAL = {Advances in Mathematics},
    VOLUME = {195},
      YEAR = {2005},
    NUMBER = {2},
     PAGES = {491--523},
      ISSN = {0001-8708,1090-2082},
       DOI = {10.1016/j.aim.2004.08.004},
       URL = {https://doi.org/10.1016/j.aim.2004.08.004},
}

@article {MR4361906,
    AUTHOR = {Livshyts, Galyna V.},
     TITLE = {The smallest singular value of heavy-tailed not necessarily
              i.i.d. random matrices via random rounding},
   JOURNAL = {J. Anal. Math.},
  FJOURNAL = {Journal d'Analyse Math\'ematique},
    VOLUME = {145},
      YEAR = {2021},
    NUMBER = {1},
     PAGES = {257--306},
      ISSN = {0021-7670,1565-8538},
       DOI = {10.1007/s11854-021-0183-2},
       URL = {https://doi.org/10.1007/s11854-021-0183-2},
}

@article {MR4904157,
    AUTHOR = {Roos, Bero},
     TITLE = {New inequalities for permanents and hafnians and some
              generalizations},
   JOURNAL = {Linear Multilinear Algebra},
  FJOURNAL = {Linear and Multilinear Algebra},
    VOLUME = {73},
      YEAR = {2025},
    NUMBER = {8},
     PAGES = {1634--1667},
      ISSN = {0308-1087,1563-5139},
       DOI = {10.1080/03081087.2024.2436061},
       URL = {https://doi.org/10.1080/03081087.2024.2436061},
}

@article {MR494478,
    AUTHOR = {Hal\'asz, G.},
     TITLE = {Estimates for the concentration function of combinatorial
              number theory and probability},
   JOURNAL = {Period. Math. Hungar.},
  FJOURNAL = {Periodica Mathematica Hungarica. Journal of the J\'anos Bolyai
              Mathematical Society},
    VOLUME = {8},
      YEAR = {1977},
    NUMBER = {3-4},
     PAGES = {197--211},
      ISSN = {0031-5303,1588-2829},
       DOI = {10.1007/BF02018403},
       URL = {https://doi.org/10.1007/BF02018403},
}

@book {MR2103758,
    AUTHOR = {Good, Phillip},
     TITLE = {Permutation, parametric and bootstrap tests of hypotheses},
    SERIES = {Springer Series in Statistics},
   EDITION = {Third},
 PUBLISHER = {Springer-Verlag, New York},
      YEAR = {2005},
     PAGES = {xx+315},
      ISBN = {0-387-20279-X},
}

@article {MR4356701,
    AUTHOR = {Jain, Vishesh and Sah, Ashwin and Sawhney, Mehtaab},
     TITLE = {Singularity of discrete random matrices},
   JOURNAL = {Geom. Funct. Anal.},
  FJOURNAL = {Geometric and Functional Analysis},
    VOLUME = {31},
      YEAR = {2021},
    NUMBER = {5},
     PAGES = {1160--1218},
      ISSN = {1016-443X,1420-8970},
       DOI = {10.1007/s00039-021-00580-6},
       URL = {https://doi.org/10.1007/s00039-021-00580-6},
}

@article {MR5042290,
    AUTHOR = {Li, Dongbin and Litvak, Alexander E. and Yu, Tingzhou},
     TITLE = {An upper bound on the smallest singular value of dense random
              combinatorial matrices},
   JOURNAL = {J. Complexity},
  FJOURNAL = {Journal of Complexity},
    VOLUME = {95},
      YEAR = {2026},
     PAGES = {Paper No. 102031},
      ISSN = {0885-064X,1090-2708},
       DOI = {10.1016/j.jco.2026.102031},
       URL = {https://doi.org/10.1016/j.jco.2026.102031},
}

\end{document}